\documentclass[12pt]{amsart}%[12pt]{book} %\usepackage{amsart}
\usepackage{amssymb}
\usepackage{latexsym}
\usepackage{amsthm}
\usepackage{amsfonts}
\usepackage{amscd}
\usepackage{amsxtra}
\usepackage{color}

\numberwithin{equation}{section}

\newtheorem{theorem}{Theorem}[section]
\newtheorem{lemma}[theorem]{Lemma}
\newtheorem{proposition}[theorem]{Proposition}
\newtheorem{corollary}[theorem]{Corollary}

\theoremstyle{definition}
\newtheorem{definition}[theorem]{Definition}

\newtheorem{question}[theorem]{Question}

\newtheorem{remark}[theorem]{Remark}

\newcommand{\N}{\mathbb{N}} %% Conjunto naturales:     \N
\newcommand{\C}{\mathbb{C}} %% Conjunto complejos:     \C
\newcommand{\K}{\mathbb{K}} %% Cuerpo		           \K
\newcommand{\A}{\mathcal{A}}

\newcommand{\abs}[1]{\left\lvert#1\right\rvert}

\newcommand{\norm}[1]{\lVert#1\rVert}

\makeatletter
\newcommand{\subalign}[1]{%
  \vcenter{%
    \Let@ \restore@math@cr \default@tag
    \baselineskip\fontdimen10 \scriptfont\tw@
    \advance\baselineskip\fontdimen12 \scriptfont\tw@
    \lineskip\thr@@\fontdimen8 \scriptfont\thr@@
    \lineskiplimit\lineskip
    \ialign{\hfil$\m@th\scriptstyle##$&$\m@th\scriptstyle{}##$\hfil\crcr
      #1\crcr
    }%
  }%
}
\makeatother

\title{Disjoint frequently hypercyclic pseudo-shifts}

\author[\"O. Martin]{\"Ozg\"ur Martin}
\address[\"O. Martin]{Department of Mathematics, Mimar Sinan Fine Arts University, Silah\c{s}\"or Cad. 71, Bomonti \c{S}i\c{s}li 34380, Istanbul (Turkey)}
\email{ozgur.martin@msgsu.edu.tr}
 \thanks{The first author was partially supported by Mimar Sinan Fine Arts University Scientific Research Project [grant no. 2018-15]}

\author[Q. Menet]{Quentin Menet}
\address[Q. Menet]{Service de Probabilit\'e et Statistique, D\'epartement de Math\'ematique\\ Universit\'{e} de Mons\\ Place du Parc 20\\ 7000 Mons (Belgium)}
\email{quentin.menet@umons.ac.be}
 \thanks{The second author is a Research Associate of the Fonds de la Recherche Scientifique - FNRS. This work was supported in part by
the project FRONT of the French
National Research Agency (grant ANR-17-CE40-0021).}

\author[Y. Puig]{Yunied Puig}
\address[Y. Puig]{Department of Mathematics, Claremont McKenna College, 850 Columbia Avenue, Claremont, CA 91711 (USA)}
\email{puigdedios@gmail.com}
\keywords{Disjoint hypercyclicity, Frequent hypercyclicity, Weighted shift}
\subjclass{47A16, 47B37}

\begin{document}

\maketitle

\begin{abstract}
	We obtain a Disjoint Frequent Hypercyclicity Criterion and show that it characterizes disjoint frequent hypercyclicity for a family of unilateral pseudo-shifts on $c_0(\N)$ and $\ell^p(\N)$, $1\le p <\infty$. As an application, we characterize disjoint frequently hypercyclic weighted shifts. We give analogous results for the weaker notions of disjoint upper frequent and reiterative hypercyclicity. Finally, we provide counterexamples showing that, although the frequent hypercyclicity, upper frequent hypercyclicity, and reiterative hypercyclicity coincide for weighted shifts on $\ell^p(\N)$, this equivalence fails for disjoint versions of these notions.
\end{abstract}

% NEW SECTION
\section{Introduction}
The aim of this paper is to study the notion of disjointness for frequently hypercyclic operators. Frequent hypercyclicity is one of the central notions in linear dynamics and it has been studied extensively since its introduction in 2006 by Bayart and Grivaux \cite{BaGr2}. On the other hand, disjointness in linear dynamics was introduced in 2007 independently by Bernal \cite{Ber} and by B\`es and Peris \cite{BePe07}, and although, disjoint hypercyclic operators have been well studied, to the best of our knowledge, no work has been appeared on disjoint frequently hypercyclic operators. We refer to the recent monographs \cite{BaMa} and \cite{GrPe} for an introduction to linear dynamics.

Weighted shift operators play an important role in linear dynamics. Recall that, for a bounded weight sequence $(w_n)_{n \geq 1}$ in $\C$, the unilateral weighted backward shift $B_w$ on $c_0(\N)$ or $\ell^p(\N)$ is defined by $B_w(e_1) = 0$ and $B_w(e_n)=w_n e_{n-1}$ for $n \geq 2$, where $(e_n)$ is the canonical basis.  Despite their simple form, with their wide-ranging (chaotic) dynamical properties, they provide a good source for examples and counterexamples (one can see \cite{Sal} or the recent papers \cite{BaRu}, \cite{BMPP1}, and \cite{BoGr} among many). Unfortunately, disjoint dynamics of hypercyclic weighted shifts are very limited. As shown in \cite{BeMaSa}, they can never be disjoint weakly mixing and, thus, never be disjoint mixing nor satisfy the Disjoint Hypercyclicity Criterion. In \cite{CoMaSa}, Çolakoğlu, Sanders and the first author characterized disjoint hypercyclicity of pseudo-shifts which are a generalization of the weighted shift operators. 

\begin{definition}
	Let $X = c_0(\N)$ or $\ell^p(\N)$, $1\le p <\infty$, $w = (w_j)_{j\in\N}$ be a weight sequence, and let $f:\mathbb{N}\to \N$ be a strictly increasing map with $f(1) > 1$. The unilateral pseudo-shift $T_{f,w}: X \to  X$ is given by
	\[T_{f,w}(\sum_{j=1}^{\infty}x_je_j)=\sum_{j=1}^{\infty}w_{f(j)}x_{f(j)}e_j,\]
where $\{e_j: j \geq 1\}$ is the canonical basis of $X$.
\end{definition}

Observe that, this family of operators includes the unilateral weighted shifts and their powers which are the source of many examples and counterexamples for disjoint hypercyclic operators (see \cite{BeMaSa}, \cite{BePe07} and \cite{MaSa16}). Indeed, weighted shifts are pseudo-shifts with the inducing map $f(n) = n +1$, $n \in \N$, and a weighted shift raised to the power $r \in \N$ is a pseudo-shift with the inducing map $f(n) = n +r$.

In Section 2, we give a Disjoint $\mathcal{A}$-Hypercyclicity Criterion which is inspired by \cite[Theorem~9]{BMPP1}, where $\mathcal{A} \subset \mathcal{P}(\N)$ is the family of sets of positive lower density, positive upper density, or positive upper Banach density. We also show that this criterion characterizes disjoint $\mathcal{A}$-hypercyclicity  for a family of unilateral pseudo-shifts on $c_0(\N)$ and $\ell^p(\N)$, $p \geq 1$, which also includes unilateral weighted shifts and their powers. 

In Section 3, using a result of Bonilla and Grosse-Erdmann \cite[Theorem 15]{BoGr}, when  $\mathcal{A}$ is the family of  sets of positive upper density or positive upper Banach density, we simplify our characterization of disjoint $\mathcal{A}$-hypercyclic pseudo-shifts and provide it for a larger family of pseudo-shifts. 

In \cite{BaRu} and \cite{BMPP1}, it was shown that the notions of frequent hypercyclicity, upper frequent hypercyclicity, and reiterative hypercyclicity coincide for weighted shifts on $\ell^p(\N)$. In Section~4, we give counterexamples showing that this is no longer true for disjoint versions of these notions. We also show that, on $\ell^p(\N)$, two reiteratively hypercyclic weighted shifts are disjoint reiteratively hypercyclic if and only if they are disjoint hypercyclic. We end the paper with some open questions.

In the rest of the Introduction, we recall the notions that we will use in the following sections. 

Let $\N$ denote the set of natural numbers, $X$ be a separable and infinite dimensional Banach space over the  real or complex scalar field $\K$, and let $\mathcal{L}(X)$ denote the algebra of continuous linear operators on $X$. An operator $T \in \mathcal{L}(X)$ is called {\it hypercyclic} if there exists $x \in X$ such that $\{T^nx:n\in \N\}$ is dense in $X$ and such a vector $x$ is said to be a hypercyclic vector for $T$. 

The lower density, the upper density, and the upper Banach density of a set $A \subset \N$ is defined by
\begin{equation*}
	\mbox{\underline{dens}} A:=\liminf_{N\to \infty}   \frac{\text{card}\{n \leq N: n\in A\}}{N}, \\
\end{equation*}
\begin{equation*}
	\overline{\text{dens}}A:=\limsup_{N\to \infty}   \frac{\text{card}\{n \leq N: n\in A\}}{N},  \mbox{and}\\
\end{equation*}
\begin{equation*}
	\overline{\text{Bd}} A:=\limsup_{N\to \infty} \max_{k \geq 1} \frac{\text{card}\{A \cap [k+1, \ldots, k+N]\}}{N}, \\
\end{equation*}
respectively. Let $\mathcal{A}_\infty$, $\mathcal{A}_{uBd}$, $\mathcal{A}_{ud}$, and $\mathcal{A}_{ld} \subset \mathcal{P}(\N)$ denote the family of infinite sets, sets of positive upper Banach density, sets of positive upper density, and sets of positive lower density, respectively. 

Let $\mathcal{A} = \mathcal{A}_\infty, \mathcal{A}_{uBd}, \mathcal{A}_{ud}$, or $\mathcal{A}_{ld} $. An operator $T \in \mathcal{L}(X)$ is called  $\mathcal{A}$-{\it hypercyclic} if there exists some $x \in X$ such that for every non-empty open set $U \subset X$ we have  $\{n : T^nx \in U\} \in \mathcal{A}$. Such a vector $x$ is called an  $\mathcal{A}$-{\it hypercyclic vector} for the operator $T$. For $\mathcal{A} = \mathcal{A}_\infty, \mathcal{A}_{uBd}, \mathcal{A}_{ud}$, or $\mathcal{A}_{ld}$, $\mathcal{A}$-hypercyclic operators are called {\it hypercyclic}, {\it reiteratively hypercyclic}, {\it upper frequently hypercyclic}, or {\it frequently hypercyclic}, respectively.

For $N \geq 2$, operators $T_1, \ldots,T_N \in \mathcal{L}(X)$ are called {\it disjoint hypercyclic} if the direct sum operator $T_1 \oplus \dots \oplus T_N$ has a hypercyclic vector of the form $(x, \ldots, x) \in X^N$.  We say $T_1, \ldots , T_N$ are {\it disjoint} $\mathcal{A}$-{\it hypercyclic} if there exists a vector $x$ in $X$ such that the vector $(x, \dots, x)$ is an $\mathcal{A}$-hypercyclic vector for the direct sum operator $T_1 \oplus  \cdots \oplus T_N$ on $X^N$. Such a vector $x$ is called a disjoint $\mathcal{A}$-hypercyclic vector for $T_1, \ldots,T_N$. For $\mathcal{A} = \mathcal{A}_\infty, \mathcal{A}_{uBd}, \mathcal{A}_{ud}$, or $\mathcal{A}_{ld}$, disjoint $\mathcal{A}$-hypercyclic operators are called {\it disjoint hypercyclic}, {\it disjoint reiteratively hypercyclic}, {\it disjoint upper frequently hypercyclic}, or {\it disjoint frequently hypercyclic}, respectively.

%NEW SECTION
\section{Disjoint $\A$-hypercyclicity and pseudo shifts}

Inspired by \cite[Theorem 9]{BMPP1}, we first give a Disjoint $\A$-Hypercyclicity Criterion as a sufficient condition for disjoint $\A$-hypercyclicity. As we will have to deal with $N$ operators, we will denote by $[N]$ the set $\{1,\dots,N\}$.

%THEOREM

\begin{theorem}[Disjoint $\A$-Hypercyclicity Criterion]\label{Ahypc}
Let $X$ be a separable Banach space,  $T_1,\dots,T_N\in \mathcal{L}(X)$ and $\mathcal{A} = \mathcal{A}_\infty$, $\A_{uBd}$, $\A_{ud}$ or $\mathcal{A}_{ld}$. If there exist a dense set $Y_0$ in $X$, a dense sequence $(y_l)_{l\ge 1}\subset Y_0^N$, $S_n:Y_0^N\rightarrow X$, $n\ge 0$, disjoint sets $A_k\in \A$, $k\ge 1$, and a sequence $(\varepsilon_k)$ of positive real numbers tending to $0$ such that for each $l\ge 1$, each $j\in [N]$,
\begin{enumerate}
\item $\sum_{n\in A_l}S_n y_l$ converges unconditionally in $X$ uniformly in $l$,
\item for any $n\in \bigcup_{k\ge 1}A_k$,
\[\Big\|\sum_{i\in A_l\backslash\{n\}}T_j^nS_iy_l\Big\|\le \varepsilon_l,\]
\item for any $n\in A_l$, any $k<l$,
\[\Big\|\sum_{i\in A_k}T_j^nS_iy_k\Big\|\le \varepsilon_l,\]
\item for any $n\in A_l$,
\[\|T_j^nS_ny_l-y_{l,j}\|\le \varepsilon_l,\]
\end{enumerate}
then $T_1,\dots,T_N$ are disjoint $\A$-hypercyclic.
\end{theorem}

\begin{proof}
Since $\sum_{n\in A_l}S_n y_l$ converges unconditionally in $X$ uniformly in $l$ and the sets $A_l$ are disjoint, we can find an increasing sequence $(l_s)_{s\ge 1}$ such that
$(y_{l_s})_s$ is dense in $Y_0^N$, such that $\sum_{s=1}^{\infty}s\varepsilon_{l_s}<\infty$ and such that for every $1\le t\le s$, every finite set $F$ with $\min F\ge \min A_{l_s}$, we have
\[
\|\sum_{n\in F\cap A_{l_t}}S_n y_{l_t}\|\le \frac{1}{s^2}.\]
Indeed, it suffices to choose $l_s$ such that $\|y_{l_s,j}-y_{s,j}\|\le \frac{1}{s}$ for every $j\in [N]$, such that $s\varepsilon_{l_s}\le \frac{1}{s^2}$ and such that if we let $M$ the smallest index satisfying  
$\|\sum_{n\in F\cap A_{l}}S_n y_{l}\|\le \frac{1}{s^2}$ for any finite set $F$ with $\min F\ge M$ and any $l\ge 1$, then $\min A_{l_s}\ge M$. 

Without loss of generality, we can thus assume that $\sum_{l=1}^{\infty}l\varepsilon_{l}<\infty$ and that for every $1\le l\le j$, every finite set $F$ with $\min F\ge \min A_{j}$, we have
\[
\|\sum_{n\in F\cap A_{l}}S_n y_{l}\|\le \frac{1}{j^2}.\]

Let $x:=\sum_{l=1}^{\infty}\sum_{n\in A_l}S_ny_l$. We show that $x$ is well-defined and that $x$ is a disjoint $\A$-hypercyclic vector for $T_1,\dots,T_N$. Let $F$ be a finite subset with $\min F\ge \min A_j$ then
\begin{align*}
\sum_{l=1}^{\infty}\Big\|\sum_{n\in F\cap A_l}S_ny_l\Big\|&\le \sum_{l=1}^{j}\Big\|\sum_{n\in F\cap A_{l}}S_ny_l\Big\|+\sum_{l=j+1}^{\infty}\Big\|\sum_{n\in F\cap A_{l}}S_ny_l\Big\|\\
&\le \sum_{l=1}^{j} \frac{1}{j^2}+ \sum_{l=j+1}^{\infty} \frac{1}{l^2}\xrightarrow[j\to \infty]{} 0.
\end{align*}

The vector $x$ is thus well-defined. On the other hand, for any $l\ge 1$, $j\in [N]$ and $n\in A_{l}$, we get
\begin{align*}
\| T_j^nx-y_{l,j}\|&\le \sum_{k=1}^{l-1}\Big\|\sum_{i\in A_k}T_j^nS_iy_k\Big\|+\sum_{k=l}^{\infty}\Big\|\sum_{i\in A_k\backslash\{n\}}T_j^nS_iy_k\Big\|+\|T_j^nS_ny_l-y_{l,j}\|\\
 &\le (l-1)\varepsilon_l+\sum_{k=l}^{\infty} \varepsilon_k+\varepsilon_l\le \sum_{k=l}^{\infty} k\varepsilon_k+\varepsilon_l\xrightarrow[l\to \infty]{} 0.
 \end{align*}
The vector $x$ is thus a disjoint $\mathcal{A}$-hypercyclic vector for $T_1,\dots,T_N$.
\end{proof}

The above criterion relies on the existence of disjoint sets $A_k\in \mathcal{A}$ with convenient separation property. We already know that for any sequence of positive integers $(N_j)$, there exists a family $(A_j)_{j\ge 1}$ of disjoint sets in $\mathcal{A}_{ld}$ such that for every $j\ge 1$, $\min(A_j)\ge N_j$ and for every $n\in A_j$, $m\in A_{j'}$ with $n\ne m$, $|n-m|\ge N_j+N_{j'}$ (see \cite{BaGr2} and \cite{BoGr07}). However we will need here a more general result allowing us to deal with other densities, with multiples and allowing us to extract the sequence $(A_j)_{j\ge 1}$ from another sequence $(B_j)_{j\ge 1}\subset \mathcal{A}$.

%PROPOSITION
%LEMMA

\begin{lemma}
\label{lemma.polansky2}
Let  $\mathcal{A} = \mathcal{A}_\infty$, $\mathcal{A}_{uBd}$, $\mathcal{A}_{ud}$ or $\mathcal{A}_{ld}$, let $(B_j)_{j\ge 1}$ be a family of sets in $\mathcal{A}$, let $(N_j)$ be a sequence of positive integers and let $Q\in \mathbb{N}$. There exists a family $(A_j)_{j\ge 1}$ of disjoint sets in $\mathcal{A}$ such that for every $j\ge 1$, $A_j\subset B_j$, $\min(A_j)\ge N_j$ and for every $n\in A_j$, $m\in A_{j'}$ with $n\ne m$, for every $1\le q,q'\le Q$, $|qn-q'm|\ge N_j+N_{j'}$.
\end{lemma}

\begin{proof} Without loss of generality, we can assume that $(N_j)$ is increasing.\\
\text{}\\
For the case $\mathcal{A} = \mathcal{A}_\infty$, a simple diagonalization argument suffices to define the family $(A_j)_{j \geq 1} = (\{a_{j,n}:n\ge 1\})_{j \geq 1}$. In step 1, we choose $a_{1,1}\in B_1$ such that $a_{1,1} \ge N_1$. In step 2, we choose $a_{1,2} \in B_1$ such that $a_{1,2} \ge Q  a_{1,1} + 2N_1$ and $a_{2,1} \in B_2$ such that $a_{2, 1} \ge Qa_{1,2} + N_1 + N_2$. Finally, assume that we have chosen $\{a_{i,j}: 1 \leq i,j \leq k-1 \mbox{ and } i+j\le k\}$ such that $a_{i,n} \ge Qa_{j,m}+ N_i + N_j$ for any $1 \leq i,j \leq k-1$ with 
$i+n\ge j+m$ or with $i+n=j+m$ and $i>j$. We can now choose $a_{1,k} \in B_1$ such that $a_{1,k} \ge Qa_{k-1,1} + N_1 + N_{k-1}$ and then, for $i = 2, \ldots, k$, in increasing order, choose $a_{i,k+1-i} \in B_i$ such that $a_{i,k+1-i} \ge Qa_{i-1, k + 2 - i}  + N_i + N_{i-1}$. Continuing this way, we will have the desired result if we define $A_j := \{a_{j,n}: n \geq 1\}$ for every $j \geq 1$.\\ 

For $\mathcal{A}_{uBd}$, we proceed in a similar way. Let $0<D_j<\overline{Bd}(B_j)$. In step 1, we choose $n_{1,1}=1$ and $a_{1,1}\in B_1$ such that $a_{1,1} \ge N_1$. In step 2, choose $a_{1,n_{1,1}+1},\dots,a_{1,n_{1,2}} \in B_1$ for some $n_{1,2}>n_{1,1}$ such that 
$a_{1,2}\ge Q a_{1,1}+2N_1$, $a_{1,n}\ge a_{1,n-1}+2N_1$ for every $n\in (n_{1,1}+1,n_{1,2}]$ and such that there exist $N\ge 2$ and $M\ge QN+2N_1$ such that
 \[\frac{\text{card}(\{a_{1,n}:n\in (n_{1,1},n_{1,2}]\}\cap \{M+1,\dots,M+N\})}{N}\ge \frac{D_1}{2N_1}.\]
  This is done by considering $N\ge 2$ and $M\ge \max\{Q a_{1,1}+2N_1,QN+2N_1\}$ satisfying
 \[\frac{\text{card}(B_1\cap \{M+1,\dots,M+N\})}{N}\ge D_1\]
 and by letting $a_{1,n}=\min\{k\in B_1\cap \{M+1,\dots,M+N\}:k\ge a_{1,n-1}+2N_1\}$ for any $n>n_{1,1}$ while $\{k\in B_1\cap \{M+1,\dots,M+N\}:k\ge a_{1,n-1}+2N_1\}$ is non-empty.
 
 Note that for any $n,m\in (n_{1,1},n_{1,2}]$ with $n\ne m$, any $1\le q,q'\le Q$, $|qa_{1,n}-q'a_{1,m}|\ge 2N_1$. Indeed, if $q=q'$, it follows from the fact that $|a_{1,n}-a_{1,m}|\ge 2N_1$ and if $q>q'$, we have
 \[qa_{1,n}-q'a_{1,m}\ge qM-q'(M+N)\ge M-QN\ge 2N_1.\]

  We then choose $a_{2,1},\dots,a_{2,n_{2,2}} \in B_2$ for some $n_{2,2}$ such that $a_{2,1}\ge Qa_{1,n_{1,2}}+N_1+N_2$ and $a_{2,n}\ge a_{2,n-1}+2N_2$ for every $n\in (1,n_{2,2}]$ and such that there exist $N\ge 2$ and $M\ge QN+2N_2$ satisfying
 \[\frac{\text{card}(\{a_{2,n}:n\in [1,n_{2,2}]\}\cap \{M+1,\dots,M+N\})}{N}\ge \frac{D_2}{2N_2}.\]
Finally, assume that we have chosen $\{a_{i,j}: 1 \leq i \leq k-1 \mbox{ and } 1 \leq j \leq n_{i,k-1}\}$ as above. We now choose $a_{1,n} \in B_1$ with $n\in (n_{1,k-1},n_{1,k}]$ such that $a_{1,n_{1,k-1}+1} \ge Qa_{k-1,n_{k-1,k-1}} + N_1 + N_{k-1}$ and $a_{1,n}\ge a_{1,n-1}+2N_1$ for every $n\in (n_{1,k-1}+1,n_{1,k}]$ so that there exist $N\ge k$ and $M\ge QN+2N_1$ such that
 \[\frac{\text{card}(\{a_{1,n}:n\in (n_{1,k-1},n_{1,k}]\}\cap \{M+1,\dots,M+N\})}{N}\ge \frac{D_1}{2N_1}.\]
 Next, for $i = 2, \ldots, k-1$, in increasing order, we choose $a_{i,n} \in B_i$ with $n\in (n_{i,k-1},n_{i,k}]$ such that $a_{i,n_{i,k-1}+1} \ge Qa_{i-1,n_{i-1,k}} + N_i + N_{i-1}$ and $a_{i,n}\ge a_{i,n-1}+2N_i$ for every $n\in (n_{i,k-1}+1,n_{i,k}]$ so that 
 there exist $N\ge k$ and $M\ge QN+2N_i$ such that
 \[\frac{\text{card}(\{a_{i,n}:n\in (n_{i,k-1},n_{i,k}]\}\cap \{M+1,\dots,M+N\})}{N}\ge \frac{D_i}{2N_i}.\]
It remains to choose $a_{k,1},\dots,a_{k,n_{k,k}}$ for some $n_{k,k}$ such that $a_{k,1}\ge Q a_{k-1,n_{k-1,k}}+N_k+N_{k-1}$ and $a_{k,n}\ge a_{k,n-1}+2N_k$ for every $n\in (1,n_{k,k}]$ and such that 
 there exist $N\ge k$ and $M\ge QN+2N_k$ satisfying
 \[\frac{\text{card}(\{a_{k,n}:n\in [1,n_{k,k}]\}\cap \{M+1,\dots,M+N\})}{N}\ge \frac{D_k}{2N_k}.\]
   Continuing this way, we will have the desired result if we define $A_j := \{a_{j,n}: n \geq 1\}$ for every $j \geq 1$.\\ 

For the case $\mathcal{A} = \mathcal{A}_{ld}$ or $\mathcal{A} = \mathcal{A}_{ud}$, we denote by $d_j$ the lower density of $B_j$ and by $D_j$ the upper density of $B_j$.  Without loss of generality, we can assume that for every $N\ge 1$,
\[\frac{|B_j\cap \{0,\dots,N\}|}{N}\le 2D_j.\]
If $B_j\in \mathcal{A}_{ld}$, we can also assume that $24Q^3D_{j+1}\le \frac{d_{j}}{2N^2_j}$ and if $B_j\in \mathcal{A}_{ud}$, we can assume that $24Q^3D_{j+1}\le \frac{D_{j}}{2N^2_j}$.

 Let $(b_{j,n})_{n\ge 1}$ be an increasing enumeration of $B_j$ and $R=\{\frac{q}{q'}:1\le q,q'\le Q\}$.  We consider
\[A_j=\{b_{j,2^{j}N^2_jn}:n\ge 1\}\backslash\{r b_{k,2^{k}N^2_k n}+l:k > j,\ n \ge 1,\ -2N_k< l< 2N_k,\ r\in R\}\]
where $l$ can be a real number. We then have $A_j\subset B_j$. Moreover, we remark that for every $N\ge 1$, we get
\begin{align*}
&\frac{\text{card}\left(\{r b_{k,2^{k}N^2_k n}+l:k > j, n \ge 1, -2N_k< l< 2N_k,\ r\in R\}\cap \{0,\dots,N\}\right)}{N}\\
&\quad\le \sum_{1\le q,q'\le Q}\sum_{k=j+1}^{\infty}\frac{\text{card}\left(\{\frac{q}{q'}b_{k,2^{k}N^2_k n}+l:n \ge 1, -2N_k< l< 2N_k\}\cap \{0,\dots,N\}\right)}{N}\\
&\quad \le \sum_{k=j+1}^{\infty}\frac{4Q^2N_k\text{card}\left(\{n \ge 1: b_{k,2^{k}N^2_k n}\le Q(N+2N_k)\}\right)}{N}\\
&\quad \le \sum_{k=j+1}^{\infty}\frac{\frac{4Q^2}{2^kN_k}\text{card}\left(\{n \ge 1: b_{k,n}\le Q(N+2N_k)\}\right)}{N}\\
&\quad \le \sum_{k=j+1}^{\infty}\frac{8Q^2}{2^kN_k}\frac{Q(N+2N_k)}{N}D_k\\
&\quad \le \sum_{k=j+1}^{\infty}\frac{8Q^3}{2^kN_k}D_k+\sum_{k=j+1}^{\infty}\frac{16Q^3}{2^kN}D_k
\le \sum_{k=j+1}^{\infty}\frac{24Q^3}{2^k}D_k.
\end{align*}
Therefore, if $B_j\in \mathcal{A}_{ld}$ then
\begin{align*}
\underline{\text{dens}}(A_j)& \ge \underline{\text{dens}}(\{b_{j,2^{j}N^2_jn}:n\ge 1\})\\
&\qquad -\overline{\text{dens}}\left(\{rb_{k,2^{k}N^2_k n}+l:k > j, n \ge 1, -2N_k\le l\le 2N_k,\ r\in R\}\cap \mathbb{N}\right) \\
& \ge \frac{1}{2^jN_j^2}d_j-\sum_{k=j+1}^{\infty}\frac{24Q^3}{2^k}D_k\\
& \ge  \frac{1}{2^jN_j^2}d_j-\frac{1}{2}\sum_{k=j+1}^{\infty}\frac{1}{2^kN_j^2}d_j>0
\end{align*}
while if $B_j\in \mathcal{A}_{ud}$ then
 \begin{align*}
\overline{\text{dens}}(A_j)& \ge \overline{\text{dens}}(\{b_{j,2^{j}N^2_jn}:n\ge 1\})\\
&\qquad -\overline{\text{dens}}\left(\{b_{k,2^{k}N^2_k n}+l:k > j, n \ge 1, -2N_k\le l\le 2N_k,\ r\in R\}\cap \mathbb{N}\right) \\
& \ge \frac{1}{2^jN_j^2}D_j-\sum_{k=j+1}^{\infty}\frac{24Q^3}{2^k}D_k\\
& \ge \frac{1}{2^jN_j^2}D_j-\frac{1}{2}\sum_{k=j+1}^{\infty}\frac{1}{2^kN_j^2}D_j>0.
\end{align*}

Moreover, $\min(A_j)\ge b_{j,2^{j}N^2_j}\ge 2^{j}N^2_j$ and if we consider $n\in A_j$, $m\in A_{j'}$ with $n\ne m$, we have:
\begin{itemize}
\item if $j=j'$ then $|n-m|\ge \inf_k(b_{j,2^{j}N^2_j(k+1)}-b_{j,2^{j}N^2_jk})\ge 2^jN^2_j\ge 2N_j$.
\item if $j>j'$ then for every $1\le q,q'\le Q$, $|qn-q'm|=q' |\frac{q}{q'}n-m|\ge 2q'N_j\ge N_j+N_{j'}$.
\end{itemize}
It remains to show that if $j=j'$, we can get that for every $1\le q\ne q'\le Q$, $|qn-q'm|\ge 2N_j$. To this end, it suffices to show that for any set $A\in \mathcal{A}_{ld}$ (resp. $A\in \mathcal{A}_{ud}$), any $r>1$ and any $N\ge 1$, it is possible to find a subset $A'\subset A$ such that 
$A'\in \mathcal{A}_{ld}$ (resp. $A'\in \mathcal{A}_{ud}$) and such that for every $n\ne m\in A'$, $|n-rm|\ge N$.\\
\text{}\\
Let $A\in  \mathcal{A}_{ld}$, $d$ the lower density of $A$, $r>1$ and $N\ge 1$. We consider $K\in \mathbb{N}$ such that $r^K-r\ge N$ and $M\in \mathbb{N}$ such that $r^{K-M}<\frac{d}{4}$ and such that for any $m\ge M$
\[\frac{\text{card}\left(A\cap [r^K,r^{m}]\right)}{r^m}\ge \frac{d}{2}.\]
 We deduce that there exists $K\le m_1< M$ such that
$\text{card}\left(A\cap [r^{m_1},r^{m_1+1}]\right)\ge \frac{d}{2}(r^{m_1+1}-r^{m_1})$.
Note that \[\frac{\text{card}\left(A\cap (r^{M+K},r^{2M}]\right)}{r^{2M}}\ge \frac{\text{card}\left(A\cap [0,r^{2M}]\right)}{r^{2M}}-\frac{\text{card}\left(A\cap [0,r^{M+K}]\right)}{r^{2M}}\ge  \frac{d}{2}-r^{K-M}\ge \frac{d}{4}.\]
There thus exists $M+K\le m_2< 2M$ such that
$\text{card}\left(A\cap [r^{m_2},r^{m_2+1}]\right)\ge \frac{d}{4}(r^{m_2+1}-r^{m_2})$. More generally, we can find a sequence $(m_k)_{k\ge 1}$ with $(k-1)M+K\le m_k< kM$ such that
\[\text{card}\left(A\cap [r^{m_k},r^{m_k+1}]\right)\ge \frac{d}{4}(r^{m_k+1}-r^{m_k}).\]
By letting $A'=A\cap\bigcup_k [r^{m_k},r^{m_k+1}-N]$, we get $A'\in \mathcal{A}_{ld}$ because for any $L\ge r^{2M}$, if $r^{(k-1)M}\le L<  r^{kM}$, we have
\begin{align*}
\frac{\text{card}\left(A'\cap[0,L]\right)}{L}
&\ge \frac{\text{card}\left(A\cap [r^{m_{k-1}},r^{m_{k-1}+1}-N]\right)}{r^{kM}}\\
&\ge \frac{d}{4}\frac{r^{m_{k-1}+1}-r^{m_{k-1}}}{r^{kM}}-\frac{N}{r^{kM}}\\
&\ge \frac{d}{4}(r-1)\frac{r^{(k-2)M}}{r^{kM}}-\frac{N}{r^{kM}}\ge \frac{d}{4}\frac{r-1}{r^{2M}}-\frac{N}{r^{kM}}
\end{align*}
and $\lim_k \frac{d}{4}\frac{r-1}{r^{2M}}-\frac{N}{r^{kM}}>0$. Moreover, for every $n\ne m\in A'$, $|n-rm|\ge N$ because if $n\in [r^{m_k},r^{m_k+1}-N]$ and $m \in [r^{m_{k'}},r^{m_{k'}+1}-N]$, we have:
\begin{itemize}
\item if $k\le k'$ then $rm-n\ge r^{m_{k'} +1}-n\ge N$;
\item if $k>k'$ then $n-rm\ge r^{m_k}-r^{m_{k-1} +2}\ge r^{(k-1)M+K}-r^{(k-1)M+1}\ge r^{(k-1)M} (r^K-r)\ge N$ by definition of $K$.
\end{itemize}

On the other hand, let $A\in \mathcal{A}_{ud}$, $D$ the upper density of $A$, $r>1$ and $N\ge 1$.  We consider $M$ such that $r^{M}-r^2\ge N$. We then divide $A$ into $M$ sets 
$A^{(m)}=A\cap\bigcup_{k\ge 1} [r^{kM+m},r^{kM+m+1}-N]$ for $m=0,\dots, M-1$. We remark that the sets $A^{(m)}$ are disjoint and that $\overline{\text{dens}}\bigcup_{m=0}^{M-1}A^{(m)}=D$. Therefore, by definition of the upper density, there exists $0\le m<M$ such that
$\overline{\text{dens}}A^{(m)}\ge \frac{D}{2M}$ and it suffices to let $A'=A^{(m)}$.
\end{proof} 

Before investigating pseudo-shifts, it is important to remark that we can extract from pseudo-shifts some weighted shifts with similar dynamical properties. Moreover, we recall that in the case of weighted shifts, we know that the notions of frequent hypercyclicity, upper frequent hypercyclicity, and reiterative hypercyclicity coincide for weighted shifts on $\ell^p(\N)$. More precisely, a weighted shift $B_w$ on $\ell^p(\N)$ satisfies one of these properties if and only if $\sum_{n\ge 2}\frac{1}{\prod_{k=2}^n|w_k|^p}<\infty$ (see \cite{BaRu} and \cite{BMPP1}).

\begin{proposition}\label{propws}
Let $T=T_{f,w}$ be an unilateral pseudo-shift on $X$ where $X=\ell^p(\mathbb{N})$ with $1\le p<\infty$ or $X=c_0(\mathbb{N})$ and $\mathcal{A} = \mathcal{A}_\infty$, $\A_{uBd}$, $\A_{ud}$ or $\mathcal{A}_{ld}$. 
If $T$ is $\mathcal{A}$-hypercyclic on $X$ then for every $j\ge 1$, $B_{v^{(j)}}$ with $v^{(j)}_n=w_{f^{n-1}(j)}$ is $\mathcal{A}$-hypercyclic on $X$.
\end{proposition}
\begin{proof}
Let $j\ge 1$ and let $\phi$ be the map defined by $\phi(x)=\sum_{k=1}^{\infty}x_{f^{k-1}(j)} e_k$. We then get that for every $x\in X$,
\[\phi T x=B_{v^{(j)}}\phi x.\]
In other words, since $\phi$ is continuous and has dense range, the operator $B_{v^{(j)}}$ is quasi-conjugate to $T$. The desired result then follows from the fact that $\mathcal{A}$-hypercyclicity is preserved under quasi-conjugacy.
\end{proof}

In order to be able to characterize disjoint $\mathcal{A}$-hypercyclic pseudo-shifts on $\ell^p(\mathbb{N})$, we need to restrict ourselves to pseudo-shifts satisfying the following condition.
 
\begin{definition}
Let $f_s:\N\to \N, s\in [N]$ be increasing maps and $\A$ be a family of subsets of $\N$. The family $\A$ is called \emph{$(f_s)_{s=1}^N$-weakly partition regular} if for any sequence $(B_l)_{l\ge 1}\subset \A$ there exists a sequence $(A_l)_{l\ge 1}\subset \A$ such that for every $l\ge 1$, $A_l\subseteq B_l$ and for every $l,k\ge 1$,
\[
f_s^m([l])\cap f_t^n([k])=\emptyset  \qquad \forall m\in A_l, \forall n\in A_k, m\neq n, \forall s, t\in [N].
\]
\end{definition}

A characterization of disjoint hypercyclic pseudo-shifts on $\ell^p(\mathbb{N})$ has already been given in \cite{CoMaSa}. Therefore, we focus on the notions of disjoint reiterative, upper frequent and frequent hypercyclicity. 

%THEOREM

\begin{theorem}[Characterization for pseudo-shifts on $\ell^p(\mathbb{N})$]\label{caraclp}
\label{theorem_charact_lp}
Let $T_1=T_{f_1,w_1},\dots,T_N=T_{f_N,w_N}$ be unilateral pseudo-shifts on $\ell^p(\mathbb{N})$ with $1\le p<\infty$ and $\mathcal{A} = \A_{uBd}$, $\A_{ud}$ or $\mathcal{A}_{ld}$. If $\A$ is $(f_s)_{s=1}^N$-weakly partition regular then $T_1,\dots,T_n$ are disjoint $\mathcal{A}$-hypercyclic if and only if
 there exist sets $A_k\in \A$, $k\ge 1$, and a sequence $(\varepsilon_k)_{k\ge 1}$ of positive real numbers tending to $0$ such that
 \begin{enumerate}
\item for any $j\ge 1$ and any $s\in [N]$, we have  $\sum_{n\ge 1}\frac{1}{|W^s_{j, n}|^p}<\infty$,

%\item for any $l, k \geq 1$ with $n\in A_l, m\in A_k$,
%for any $t, s\in [N]$ with $t\neq s$, and any $j\in f^m_t([k])\cap f_s^n((l, \infty))$
%\[
%\frac{\abs{W^s_{f_s^{-n}(j),n}}}{\abs{W^t_{f_t^{-m}(j),m}}}\le \min\{\varepsilon_l, \varepsilon_k\},
%\]

\item for any $1\le j\le l$, any $k\ge 1$, any $n\in A_k$, any $s\ne t\in[N]$,
\[\sum_{m\in A_l\backslash\{n\}: f_t^m(j)\in f_s^n((k,+\infty))} \frac{\abs{W^s_{f_s^{-n}(f_t^m(j)),n}}^p}{\abs{W^t_{j,m}}^p}\le \varepsilon_l,\]
\item for any $1\le j\le k<l$, any $n\in A_l$, any $s\ne t\in [N]$
\[\sum_{m\in A_k:f_t^m(j)\in f_s^n((l,+\infty))} \frac{\abs{W^s_{f_s^{-n}(f_t^{m}(j)),n}}^p}{\abs{W^t_{j,m}}^p}\le \varepsilon_l,\]
\item for any $l\ge 1$, any $(a_{s,j})_{s\in[N],j\in [l]}$ with $a_{s,j}\ne 0$, any $\varepsilon>0$, there exists $L\ge l$ such that for any $s\ne t\in[N]$, any $n\in A_L$, 
\begin{enumerate}
\item for any $j\in f^n_t([l])\cap f_s^n([l])$, 
\[
\left|\frac{W^s_{f_s^{-n}(j),n}}{W^t_{f_t^{-n}(j),n}}-\frac{a_{s,f_s^{-n}(j)}}{a_{t,f_t^{-n}(j)}}\right|\le \varepsilon
\]
\item for any $j\in f^{n}_t([l])\cap f_s^{n}((l, \infty))$
\[
\frac{\abs{W^s_{f_s^{-n}(j),n}}}{\abs{W^t_{f_t^{-n}(j),n}}}\le \varepsilon,
\]
\end{enumerate}
\end{enumerate}
where $W^s_{l, n}=\prod_{m=1}^nw_{s, f_s^m(l)}$, for all $s\in [N]$.
\end{theorem}
\begin{proof}
Suppose that $T_1,\dots,T_N$ possess a disjoint $\mathcal{A}$-hypercyclic vector $x$. Let $j\ge 1$ and $s\in [N]$. From Proposition~\ref{propws}, we deduce that
the unilateral weighted shift $B_{v^s_{j}}$ with $v^s_{j,n}=w_{s,f_s^{n-1}(j)}$ is $\mathcal{A}$-hypercyclic on $\ell^p(\mathbb{N})$ and it follows from the characterization of $\mathcal{A}$-hypercyclic weighted shifts on $\ell^p(\mathbb{N})$ that
\[\sum_{n\ge 1}\frac{1}{|W^s_{j, n}|^p}<\infty.\]
We now consider a dense sequence $(a^l)_{l\ge 1}\subset \ell^p(\mathbb{N})^N$ such that for every $s\in [N]$, 
\[\text{deg}(a^l_s)\le l\quad\text{and}\quad a^{l}_{s,j}\ne 0 \quad\text{for every $j\le l$.}\]
Let $C_l=\frac{l}{\min_{s\in [N],j\in [l]}|a^l_{s,j}|}$. We then let $A_l=\{n\ge 0:\|T_s^nx-C_l a^l_s\|<\tau_l\ \text{for every $s\in[N]$}\}$ where $(\tau_l)_{l\ge 1}\subset  (0,1)$ decreases to $0$.

It follows that for every $1\le j\le k$, every $s\in [N]$ and every $n\in A_k$,
\[|W^s_{j,n}x_{f_s^n(j)}-C_k a^k_{s,j}|<\tau_k\]
and thus 
\[0<k-\tau_k \le C_k |a^k_{s,j}|-\tau_k\le |W^s_{j,n}x_{f_s^n(j)}|.\]
In particular, for every $1\le j\le k$, every $s\in [N]$ and every $n\in A_k$, $x_{f_s^n(j)}\ne 0$.
On the other hand, for every $k\ge 1$, every $s\in [N]$ and every $n\in A_k$,
\[\sum_{j>k}|W^s_{j,n}x_{f_s^n(j)}|^p<\tau^p_k.\]
Therefore, for any $1\le j\le l$, any $k\ge 1$, any $n\in A_k$, any $s\ne t\in[N]$,
\begin{align*}
\sum_{m\in A_l\backslash\{n\}: f_t^m(j)\in f_s^n((k,+\infty))}  \frac{\abs{W^s_{f_s^{-n}(f_t^m(j)),n}}^p}{\abs{W^t_{j,m}}^p}&\le 
\sum_{m\in A_l\backslash\{n\}: f_t^m(j)\in f_s^n((k,+\infty))} \frac{\abs{W^s_{f_s^{-n}(f_t^m(j)),n}x_{f_{t}^m(j)}}^p}{\abs{W^t_{j,m}x_{f_t^m(j)}}^p}\\
&\le \frac{\tau_k^p}{(l-\tau_l)^p}
\le \frac{\tau_1^p}{(l-\tau_l)^p}.
\end{align*}

In the same way, for any $1\le j\le k<l$, any $n\in A_l$ and any $s\ne t\in [N]$, we have

\begin{align*}
\sum_{m\in A_k:f_t^m(j)\in f_s^n((l,+\infty))} \frac{\abs{W^s_{f_s^{-n}(f_t^{m}(j)),n}}^p}{\abs{W^t_{j,m}}^p}
&\le \frac{\tau_l^p}{(k-\tau_k)^p}\le \frac{\tau_l^p}{(1-\tau_1)^p}.
\end{align*}

We can thus consider $\varepsilon_l=\max\{\frac{\tau_1^p}{(l-\tau_l)^p}, \frac{\tau_l^p}{(1-\tau_1)^p}\}$ in order to get the conditions (2) and (3).

Let $l\ge 1$, $(b_{s,j})_{s\in[N],j\in [l]}$ with $\gamma=\min_{s\in[N],j\in [l]}\abs{b_{s,j}}> 0$, $\Gamma=\max_{s\in[N],j\in [l]}\abs{b_{s,j}}$ and $\varepsilon>0$. There exists $L\ge \max\{l,3+\frac{1}{\varepsilon},\frac{8\Gamma \tau_1}{\varepsilon \gamma}\}$ such that for every $s\in [N]$, every $j\in [L]$, \[|b_{s,j}-a^L_{s,j}|<\min\{\frac{\varepsilon\gamma^2}{32\Gamma},\frac{\gamma}{2}\},\] where $b_{s, j}=0$ for all $s\in [N]$ and $j\ge l+1$. 

In particular, we have $\min_{s\in [N],j\in [l]}|a^L_{s,j}|\in [\frac{\gamma}{2},\frac{3\gamma}{2}]$ and thus $C_L\ge \frac{2L}{3\gamma}$.
Therefore, for every $s\ne t\in[N]$, any $n\in A_L$, any $j\in f^n_t([l])\cap f_s^n([L])$, 
\begin{align*}
&\left|\frac{W^s_{f_s^{-n}(j),n}}{W^t_{f_t^{-n}(j),n}}-\frac{b_{s,f_s^{-n}(j)}}{b_{t,f_t^{-n}(j)}}\right|
\\
&\quad= \left|\frac{W^s_{f_s^{-n}(j),n}x_j}{W^t_{f_t^{-n}(j),n}x_j}-\frac{b_{s,f_s^{-n}(j)}}{b_{t,f_t^{-n}(j)}}\right|\\
&\quad\le  \frac{|W^s_{f_s^{-n}(j),n}x_jb_{t,f_t^{-n}(j)}-W^t_{f_t^{-n}(j),n}x_jb_{s,f_s^{-n}(j)}|}{|W^t_{f_t^{-n}(j),n}x_j|.|b_{t,f_t^{-n}(j)}|}\\
&\quad\le  \frac{|b_{t,f_t^{-n}(j)}|.|W^s_{f_s^{-n}(j),n}x_j-C_La^L_{s,f_s^{-n}(j)}|}{(L-\tau_L)\gamma}+\frac{C_L|b_{t,f_t^{-n}(j)}|.|a^L_{s,f_s^{-n}(j)}-b_{s,f_s^{-n}(j)}|}{(C_L |a^L_{t,f_t^{-n}(j)}|-\tau_L)\gamma}\\
&\quad\quad+
\frac{C_L|b_{s,f_s^{-n}(j)}| |b_{t,f_t^{-n}(j)}-a^L_{t,f_t^{-n}(j)}|}{(C_L |a^L_{t,f_t^{-n}(j)}|-\tau_L)\gamma}
+
\frac{
|b_{s,f_s^{-n}(j)}||C_La^L_{t,f_t^{-n}(j)}-W^t_{f_t^{-n}(j),n}x_j|}{(L-\tau_L)\gamma}\\
&\quad\le \frac{2\Gamma\tau_L}{(L-\tau_L)\gamma}+\frac{2 C_L \Gamma \frac{\varepsilon\gamma^2}{32\Gamma}}{(C_L \frac{\gamma}{2}-\tau_L)\gamma}\\
&\quad\le 
\frac{4\Gamma \tau_1}{L\gamma}+\frac{\varepsilon \gamma}{16(\frac{\gamma}{2}-\frac{\tau_L}{C_L})}
\le \frac{\varepsilon}{2}+\frac{\varepsilon \gamma}{16(\frac{\gamma}{2}-\frac{3\gamma}{2L})}
\le \frac{\varepsilon}{2}+\frac{\varepsilon \gamma}{16(\frac{\gamma}{2}-\frac{3\gamma}{8})}\le \varepsilon.
\end{align*} 
The condition (a) is thus satisfied. Moreover, we deduce that for every $s\neq t\in [N]$, any $n\in A_L$, any $j\in f^n_t([l])\cap f^n_s([l+1, L])$, we have \[\abs{\frac{W^s_{f_s^{-n}(j),n}}{W^t_{f_t^{-n}(j),n}}}\le \varepsilon.\]        
Furthermore, for any $j\in f^{n}_t([l])\cap f_s^{n}((L, \infty))$, we have
\[
\frac{\abs{W^s_{f_s^{-n}(j),n}}}{\abs{W^t_{f_t^{-n}(j),n}}}=\frac{\abs{W^s_{f_s^{-n}(j),n}x_j}}{\abs{W^t_{f_t^{-n}(j),n}x_j}}\le \frac{\tau_L}{L-\tau_L}<\varepsilon
\]
and we conclude that the condition (b) is also satisfied.\\ 

%On the other hand, if we consider $n\in A_l, m\in A_k$ with $m>n$, $s\in [N]$ and $j\in f^m_s([k])\cap f_s^n((l, \infty))$, we have $\|T_s^nx-a^l_s\|<\tau_l$ and thus $|W^s_{f_s^{-n}(j),n}x_j-a^l_{s,f_s^{-n}(j)}|=|W^s_{f_s^{-n}(j),n}x_j|<\tau_l$ since $f_s^{-n}(j)>l$. On the other hand, we have $\|T_s^mx-a^k_s\|<\tau_k$ and thus $|W^s_{f_s^{-m}(j),m}x_j-a^k_{s,f_s^{-m}(j)}|<\tau_k$. It follows that
%$|x_j|>0$ and that $|W^s_{f_s^{-m}(j),m} x_j|>\rho_k-\tau_k$ and thus
%\[
%\frac{\abs{W^s_{f_s^{-n}(j),n}}}{\abs{W^s_{f_s^{-m}(j),m}}}<\frac{\tau_l}{\rho_k-\tau_k}
%\]
\text{}\\
We now prove the other implication. We thus assume that  there exist sets $B_k\in \A$, $k\ge 1$, and a sequence $(\rho_k)_{k\ge 1}$ of positive real numbers tending to $0$ such that $(1)-(4)$ are satisfied. Since $\mathcal{A}$ is $(f_s)_{s=1}^N$-weakly partition regular and thanks to Lemma~\ref{lemma.polansky2}, we can consider a family of disjoint sets $(A_k)_{k\ge 1}\subset \A$ such that for every $l,k\ge 1$, $A_k\subset B_k$ and
\[
f_s^m([l])\cap f_t^n([k])=\emptyset  \qquad \forall m\in A_l, \forall n\in A_k, m\neq n, \forall s, t\in [N]
\]
and such that for every $k\ge 1$, $\min(A_k)\ge N_k+k+1$ and for every $n\in A_l$, $m\in A_{k}$ with $n\ne m$, $|n-m|\ge N_l+N_k+l+k+2$ where $N_k$ is chosen so that
for any $j\le k$ and any $s\in [N]$, we have  $\sum_{n\ge N_k}\frac{1}{|W^s_{j, n}|^p}<\rho_k$.
The sequence $(A_k)_{k\ge 1}$ then satisfies the conditions $(2)-(4)$. Moreover, in view of the separation property of $(A_k)$, for any $n\in A_k$, any $m\in A_l$ with $n\ne m$, any $j\le l$, the condition $f_t^m(j)\in f_s^n((k,+\infty))$ is equivalent to $f_t^m(j)\in f_s^n([1,+\infty))$.  The conditions (2) and (3) will thus be still satisfied if we consider a subsequence of $(A_k)$.

Let $(y_l)_{l\ge 1}\subset c_{00}^N$ dense in $\ell^p(\mathbb{N})^N$ such that for every $l\ge 1$, every $s\in [N]$, $\deg(y^s_l)\le l$, $\|y^s_l\|^p\le l$ and for every $j\le l$, $y^s_{l,j}\ne 0$, where $y_l=(y^s_l)_{s\in [N]}$. We can then assume, even if it means considering a subsequence, that the sequence $(Nl^2\rho_l)_l$ tends to $0$, that 
for every $s\ne t\in[N]$, any $n\in A_l$, any $j\in f^n_t([l])\cap f_s^n([l])$, 
\[
\left|\frac{W^s_{f_s^{-n}(j),n}}{W^t_{f_t^{-n}(j),n}}-\frac{y^s_{l,f_s^{-n}(j)}}{y^t_{l,f_t^{-n}(j)}}\right|^p\le \frac{1}{l^3}.
\]
and that for any $j\in f^{n}_t([l])\cap f_s^{n}((l, +\infty))$,
\[
\frac{\abs{W^s_{f_s^{-n}(j),n}}^p}{\abs{W^t_{f_t^{-n}(j),n}}^p} \le \frac{1}{l^3}.
\]

We can now show that $T_1,\dots,T_N$ satisfy the Disjoint $\mathcal{A}$-Hypercyclicity Criterion for $(y_l)_{l\ge 1}$, $(A_l)_{l\ge 1}$, $\varepsilon^p_l=\max\{Nl^2\rho_l,\frac{N+1}{l}\}$ and for every $l\ge 1$, every $n\ge 1$, 
\[S_n(y_l)=\sum_{t=1}^{N}\sum_{j\in J_{n,t,l}}\frac{1}{W^t_{j,n}}y^{t}_{l, j}e_{f_t^n(j)} \]
where $J_{n,1,l}=[l]$ and $J_{n,t+1,l}=[l]\backslash f_{t+1}^{-n}(\bigcup_{s\le t}f_{s}^n([l]))$ for every $t\in [N-1]$. 

%In particular, for every $l,n\ge 1$, the sets $f^n_1(J_{n,1,l}),\dots, f^n_N(J_{n,N,l})$ are disjoint.

\begin{itemize}
\item Let $l\ge 1$ and $s\in [N]$. By definition of sets $J_{n,t,l}$ and our assumptions on $A_l$, we get that for every $m\in A_l$, $n\in A_k$, if $m\ne n$ or $t\ne s\in [N]$ then 
 \[f^m_t(J_{m,t,l})\cap f^n_{s}(J_{n,s,k})=\emptyset.\]
 Let $\varepsilon>0$ and $k\ge 1$ such that $Nj^2\rho_j\le \varepsilon$ for every $j\ge k$. We have for every finite set $F\subset [N_k,+\infty)$ and every $l\le k$,
\begin{align*}
\|\sum_{n\in A_l\cap F}S_n y_l\|^p
&=\sum_{n\in A_l\cap F}\sum_{t=1}^{N}\sum_{j\in J_{n,t,l}}\frac{1}{|W^t_{j,n}|^p}|y^{t}_{l,j}|^p\\
&\le Nl^2\max_{t\in [N], j\in [l]}\left(\sum_{n\ge N_k}\frac{1}{|W^t_{j,n}|^p}\right)\le 
Nk^2\rho_k\le \varepsilon.
\end{align*}
and for every $l>k$,
\begin{align*}
\|\sum_{n\in A_l\cap F}S_n y_l\|^p
&=\sum_{n\in A_l\cap F}\sum_{t=1}^{N}\sum_{j\in J_{n,t,l}}\frac{1}{|W^t_{j,n}|^p}|y^{t}_{l,j}|^p\\
&\le Nl^2\max_{t\in [N], j\in [l]}\left(\sum_{n\ge N_l}\frac{1}{|W^t_{j,n}|^p}\right)\le 
Nl^2\rho_l\le \varepsilon.
\end{align*}
We conclude that $\sum_{n\in A_l}S_n y_l$ converges unconditionally in $\ell^p(\mathbb{N})$ uniformly in $l$.

\item Let $k,l\ge 1$ and $s\in [N]$. For any $n\in A_k$, we have
\begin{align*}
\Big\|\sum_{i\in A_l\backslash\{n\}}T_s^nS_iy_l\Big\|^p
&=\Big\|\sum_{i\in A_l\backslash\{n\}}\sum_{t=1}^{N}\sum_{j\in J_{i,t,l}}\frac{1}{W^t_{j,i}}y^{t}_{l,j}T_s^ne_{f_t^i(j)}\Big\|^p\\ 
%&=\Big\|\sum_{i\in A_l\backslash\{n\}}\sum_{t=1}^{N}\sum_{j\in J_{i,t,l}: f_t^i(j)\in f_s^n([1,++\infty))}\frac{W^s_{f_s^{-n}(f_t^i(j)),n}}{W^t_{j,i}}y^{t}_{l,j}e_{f_s^{-n}(f_t^i(j))}\Big\|\\ 
%&=\Big\|\sum_{i\in A_l\backslash\{n\}}\sum_{t=1}^{N}\sum_{j\in J_{i,t,l}: f_t^i(j)\in f_s^n((k,++\infty))}\frac{W^s_{f_s^{-n}(f_t^i(j)),n}}{W^t_{j,i}}y^{t}_{l,j}e_{f_s^{-n}(f_t^i(j))}\Big\|\\
%&\le  N l \max_{t\in [N], j\in [l]}\Big\|\sum_{i\in A_l\backslash\{n\}: f_t^i(j)\in f_s^n((k,+\infty))}\frac{W^s_{f_s^{-n}(f_t^i(j)),n}}{W^t_{j,i}}y^{t}_{l,j}e_{f_s^{-n}(f_t^i(j))}\Big\|\\
&\le N (\max_{t\in [N]}\|y^{t}_{l}\|^p)\max_{t\in [N]}\left(\sum_{i\in A_l\backslash\{n\}}\sum_{j\in J_{i,t,l}: f_t^i(j)\in f_s^n([1,+\infty))}\frac{|W^s_{f_s^{-n}(f_t^i(j)),n}|^p}{|W^t_{j,i}|^p}\right)\\
&\le Nl \max_{t\in [N]}\left(\sum_{i\in A_l\backslash\{n\}}\sum_{j\in [l]: f_t^i(j)\in f_s^n((k,+\infty))}\frac{|W^s_{f_s^{-n}(f_t^i(j)),n}|^p}{|W^t_{j,i}|^p}\right)\\
&\le Nl^2\max_{t\in [N],j\in [l]}\left(\sum_{i\in A_l\backslash\{n\}: f_t^i(j)\in f_s^n((k,+\infty))}\frac{|W^s_{f_s^{-n}(f_t^i(j)),n}|^p}{|W^t_{j,i}|^p}\right)\\
&\le Nl^2\rho_l\le \varepsilon^p_l.
\end{align*}
The fact that $\sum_{i\in A_l\backslash\{n\}: f_t^i(j)\in f_s^n((k,+\infty))}\frac{|W^s_{f_s^{-n}(f_t^i(j)),n}|^p}{|W^t_{j,i}|^p}\le \rho_l$ for $t=s$ follows from our assumption that
for any $j\in [l]$ and any $s\in [N]$, we have  $\sum_{n\ge N_l}\frac{1}{|W^s_{j, n}|^p}<\rho_l$. Indeed, if $i\in A_l\backslash\{n\}$, $n\in A_k$ and $f_s^i(j)\in f_s^n((k,+\infty))$ with $j\in [l]$ then $i>n$ and thus $i>n+N_l$ because if $i< n$, we have $j\ge \min f_s^{n-i}((k,+\infty))\ge n-i>l$. It follows that   for every $j\in [l]$,
\begin{align*}
\sum_{i\in A_l\backslash\{n\}: f_s^i(j)\in f_s^n((k,+\infty))}\frac{|W^s_{f_s^{-n}(f_s^i(j)),n}|^p}{|W^s_{j,i}|^p}&\le 
\sum_{i>n+N_l}\frac{|W^s_{f_s^{-n}(f_s^i(j)),n}|^p}{|W^s_{j,i}|^p}\\
&=\sum_{i>n+N_l}\frac{1}{|W^s_{j,i-n}|^p}\le \rho_l.
\end{align*}

\item Let $l\ge 1$ and $s\in [N]$. For any $n\in A_l$, any $k<l$, we have
\begin{align*}
\Big\|\sum_{i\in A_k}T_s^nS_iy_k\Big\|^p&=\Big\|\sum_{i\in A_k}\sum_{t=1}^{N}\sum_{j\in J_{i,t,k}}\frac{1}{W^t_{j,i}}y^{t}_{k,j}T_s^ne_{f_t^i(j)}\Big\|^p \\
&\le N(\max_{t\in [N]}\|y^{t}_{k}\|^p)\max_{t\in [N]}\left(\sum_{i\in A_k}\sum_{j\in [k]:f_t^i(j)\in f_s^n([1,+\infty))}\frac{|W^s_{f_s^{-n}(f_t^i(j)),n}|^p}{|W^t_{j,i}|^p}\right)\\
&\le  Nl^2\max_{t\in [N],j\in [k]}\left(\sum_{i\in A_k:f_t^i(j)\in f_s^n((l,+\infty))}\frac{|W^s_{f_s^{-n}(f_t^i(j)),n}|^p}{|W^t_{j,i}|^p}\right)\\
&\le \varepsilon^p_l.
\end{align*}
As above, the case $s=t$ follows from the fact that for any $k<l$, any $j\in [k]$ and any $s\in [N]$, we have  $\sum_{n\ge N_l}\frac{1}{|W^s_{j, n}|^p}<\rho_l$.

\item Let $l\ge 1$ and $s\in [N]$. For any $n\in A_l$, we get, by definition of $J_{n,t,l}$,
\begin{align*}
\|T_s^nS_ny_l-y^s_{l}\|^p&=
\|\sum_{t=1}^{N}\sum_{j\in J_{n,t,l}}\frac{1}{W^t_{j,n}}y^{t}_{l,j}T_s^ne_{f_t^n(j)}-\sum_{j\in [l]}y_{l,j}^se_j\|^p\\
&=\|\sum_{t=1}^{N}\sum_{j\in J_{n,t,l}: f_t^n(j)\in f_s^n([1,+\infty))}\frac{W^s_{f^{-n}_s(f_t^n(j)),n}}{W^t_{j,n}}y^{t}_{l,j}e_{f^{-n}_s(f_t^n(j))}-\sum_{j\in [l]}y_{l,j}^se_j\|^p\\
&=\|\sum_{t=1}^{s}\sum_{j\in J_{n,t,l}: f_t^n(j)\in f_s^n([l])}\frac{W^s_{f^{-n}_s(f_t^n(j)),n}}{W^t_{j,n}}y^{t}_{l,j}e_{f^{-n}_s(f_t^n(j))}-\sum_{j\in [l]}y_{l,j}^se_j\|^p\\
&\quad+\sum_{t\in [N]\backslash\{s\}}\sum_{j\in J_{n,t,l}: f_t^n(j)\in f_s^n((l,+\infty))}\frac{|W^s_{f^{-n}_s(f_t^n(j)),n}|^p}{|W^t_{j,n}|^p}|y^{t}_{l,j}|^p.
\end{align*}
If $i\in [l]$ then there are unique $t_i\le s$ and $j_i\in J_{n,t_i,l}$ such that
$f_s^{-n}(f^n_{t_i}(j_i))=i$. It follows that
\begin{align*}
&\|T_s^nS_ny_l-y^s_{l}\|^p\\
 &\quad\le \left\|\sum_{i\in [l]}
\Big(\frac{W^s_{i,n}}{W^{t_i}_{j_i,n}}y^{t_i}_{l,j_i}e_{i}-y_{l,i}^se_i\Big)\right\|^p+(\max_{t\in [N]}\|y_l^t\|^p)\sum_{t\in [N]\backslash\{s\}}\sum_{i\in f_s^n((l,+\infty)) \cap f_t^n([l])}\frac{|W^s_{f^{-n}_s(i),n}|^p}{|W^t_{f^{-n}_t(i),n}|^p}\\
&\quad\le l\max_{i\in [l]}
\left|\frac{W^s_{i,n}}{W^{t_i}_{j_i,n}}y^{t_i}_{l,j_i}-y_{l,i}^s\right|^p + l^2N\max_{t\in [N]\backslash\{s\}}\max_{i\in f_s^n((l,+\infty)) \cap f_t^n([l])}\frac{|W^s_{f^{-n}_s(i),n}|^p}{|W^t_{f^{-n}_t(i),n}|^p}\\
&\quad\le l(\max_{t\in[N]}\|y^{t}_{l}\|^p)\max_{i\in [l]} \left|\frac{W^s_{i,n}}{W^{t_i}_{j_i,n}}-\frac{y_{l,i}^s}{y^{t_i}_{l,j_i}}\right|^p+\frac{N}{l}\\
&\quad\le l^2\max_{t\le s}\max_{k\in f^n_t([l])\cap f^n_s([l])} \left|\frac{W^s_{f_s^{-n}(k),n}}{W^{t}_{f_t^{-n}(k),n}}-\frac{y_{l,f_s^{-n}(k)}^s}{y^{t}_{l,f_t^{-n}(k)}}\right|^p+\frac{N}{l}\\
&\quad\le \frac{1}{l}+\frac{N}{l}.
\end{align*}
\end{itemize}
\end{proof}

In the case of $c_0(\mathbb{N})$, we know from \cite{BaRu} that the characterization of $\mathcal{A}$-hypercyclic weighted shifts is more complicated than in the case of $\ell^p(\mathbb{N})$. Similar arrangements will be necessary for the condition (1).

%THEOREM

\begin{theorem}[Characterization for pseudo-shifts on $c_0(\mathbb{N})$]\label{caracc0}
Let $T_1=T_{f_1,w_1},\dots,T_N=T_{f_N,w_N}$ be unilateral pseudo-shifts on $c_0(\mathbb{N})$ and $\mathcal{A} = \A_{\infty}$, $\A_{uBd}$, $\A_{ud}$ or $\mathcal{A}_{ld}$. If $\A$ is $(f_s)_{s=1}^N$-weakly partition regular then $T_1,\dots,T_n$ are disjoint $\mathcal{A}$-hypercyclic if and only if
 there exist sets $A_k\in \A$, $k\ge 1$, and a sequence $(\varepsilon_k)$ of positive real numbers tending to $0$ such that
 \begin{enumerate}
\item for any $1\le j\le k$ and any $s\in [N]$, we have  
\[\lim_{n\in A_k} |W^s_{j, n}|=\infty\]
and for any $l, k \geq 1$, any $n\in A_l, m\in A_k$,
any $s\in [N]$ and any $j\in f^m_s([k])\cap f_s^n((l, +\infty))$,
\[
\frac{\abs{W^s_{f_s^{-n}(j),n}}}{\abs{W^s_{f_s^{-m}(j),m}}}\le \min\{\varepsilon_l, \varepsilon_k\},
\]
\item for any $1\le j\le l$, any $k\ge 1$, any $n\in A_k$, any $s\ne t\in[N]$, any $m\in A_l\backslash\{n\}$, if  $f_t^m(j)\in f_s^n((k,+\infty))$ then
\[\frac{\abs{W^s_{f_s^{-n}(f_t^m(j)),n}}}{\abs{W^t_{j,m}}}\le \varepsilon_l,\]
\item for any $1\le j\le k<l$, any $n\in A_l$, any $s\ne t\in [N]$, any $m\in A_k$, if $f_t^m(j)\in f_s^n((l,+\infty))$ then
\[\frac{\abs{W^s_{f_s^{-n}(f_t^{m}(j)),n}}}{\abs{W^t_{j,m}}}\le \varepsilon_l,\]
\item for any $l$, any $(a_{s,j})_{s\in[N],j\in [l]}$ with $a_{s,j}\ne 0$, any $\varepsilon>0$, there exists $L\ge l$ such that for every $s\ne t\in[N]$, any $n\in A_L$, 
\begin{enumerate}
\item for every $j\in f^n_t([l])\cap f_s^n([l])$, 
\[
\left|\frac{W^s_{f_s^{-n}(j),n}}{W^t_{f_t^{-n}(j),n}}-\frac{a_{s,f_s^{-n}(j)}}{a_{t,f_t^{-n}(j)}}\right|\le \varepsilon
\]
\item for every $j\in f^{n}_t([l])\cap f_s^{n}((l, \infty))$
\[
\frac{\abs{W^s_{f_s^{-n}(j),n}}}{\abs{W^t_{f_t^{-n}(j),n}}}\le \varepsilon,
\]
\end{enumerate}
\end{enumerate}
where $W^s_{l, n}=\prod_{\nu=1}^nw_{s, f_s^\nu(l)}$, for all $s\in [N]$.
\end{theorem}
\begin{proof}
Suppose that $T_1,\dots,T_N$ possess a disjoint $\mathcal{A}$-hypercyclic vector $x$.
We consider a dense sequence $(a^l)_{l\ge 1}\subset c_0(\mathbb{N})^N$ such that for every $s\in [N]$, 
\[\text{deg}(a^l_s)\le l\quad\text{and}\quad a^{l}_{s,j}\ne 0 \quad\text{for every $j\le l$.}\]
Let $C_l=\frac{l}{\min_{s\in [N],j\in [l]}|a^l_{s,j}|}$. We then let $A_l=\{n:\|T_s^nx-C_l a^l_s\|<\tau_l\ \text{for every $s\in[N]$}\}$ where $(\tau_l)\subset  (0,1)$ decreases to $0$.

It follows that for every $1\le j\le k$, every $s\in [N]$ and every $n\in A_k$,
\[|W^s_{j,n}x_{f_s^n(j)}-C_k a^k_{s,j}|<\tau_k\]
and thus 
\[0<k-\tau_k\le |W^s_{j,n}x_{f_s^n(j)}|.\]
In particular, for every $1\le j\le k$, every $s\in [N]$ and every $n\in A_k$, $x_{f_s^n(j)}\ne 0$ and
\[\lim_{n\in A_k }|W^s_{j,n}|\ge \lim_{n\in A_k }\frac{k-\tau_k}{|x_{f_s^n(j)}|}=\infty.\]
On the other hand, for every $s\in [N]$ and every $n\in A_k$,
\[\max_{j>k}|W^s_{j,n}x_{f_s^n(j)}|<\tau_k.\]

It follows that for any $l,k\ge 1$, any $n\in A_l$, $m\in A_k$, any $s\in [N]$ and any $j\in f_s^m([k])\cap f_s^n((l,+\infty))$,
\[\frac{|W^s_{f_s^{-n}(j),n}|}{|W^s_{f_s^{-m}(j),m}|}=\frac{|W^s_{f_s^{-n}(j),n}x_j|}{|W^s_{f_s^{-m}(j),m}x_j|}\le \frac{\tau_l}{k-\tau_k}.\]

and for any $1\le j\le l$, any $k\ge 1$, any $n\in A_k$, any $s\ne t\in[N]$, any $m\in A_l\backslash\{n\}$ such that $f_t^m(j)\in f_s^n((k,+\infty))$, we get
\[
\frac{\abs{W^s_{f_s^{-n}(f_t^m(j)),n}}}{\abs{W^t_{j,m}}}=
\frac{\abs{W^s_{f_s^{-n}(f_t^m(j)),n}x_{f_{t}^m(j)}}}{\abs{W^t_{j,m}x_{f_t^m(j)}}}
\le \frac{\tau_k}{(l-\tau_l)}
\le \frac{\tau_1}{(l-\tau_l)}.
\]

Moreover, for any $1\le j\le k<l$, any $n\in A_l$ and any $s\ne t\in [N]$,  we have that for all $m\in A_k$ with $f_s^{-n}(f_t^m(j))>l$, the following holds

\[
\frac{\abs{W^s_{f_s^{-n}(f_t^{m}(j)),n}}}{\abs{W^t_{j,m}}}= \frac{\abs{W^s_{f_s^{-n}(f_t^{m}(j)),n} x_{f_t^m(j)}}}{\abs{W^t_{j,m}x_{f_t^m(j)}}}
\le \frac{\tau_l}{(k-\tau_k)}\le \frac{\tau_l}{(1-\tau_1)}.
\]

We can thus consider $\varepsilon_l=\max\{\frac{\tau_1}{(l-\tau_l)}, \frac{\tau_l}{(1-\tau_1)}\}$ in order to get (1), (2), (3) and (4) is obtained as in the case of $\ell^p(\mathbb{N})$.\\
\text{}\\
We now prove the other implication and thus assume that there exist disjoint sets $B_k\in \A$, $k\ge 1$, and a sequence $(\rho_k)$ of positive real numbers tending to $0$ such that $(1)-(4)$ are satisfied. We first consider a family $(A_k)\subset \A$ such that $A_k\subset B_k$,
\[
f_s^m([l])\cap f_t^n([k])=\emptyset  \qquad \forall m\in A_l, \forall n\in A_k, m\neq n, \forall s, t\in [N]
\]
and such that for every $k\ge 1$, $\min(A_k)\ge N_k+k+1$ and for every $n\in A_l$, $m\in A_{k}$ with $n\ne m$, $|n-m|\ge N_l+N_k+l+k+2$ where $N_k$ is chosen so that
for any $j\le l \le k$ and any $s\in [N]$, we have  $\max_{n\in B_l\cap [N_k,+\infty)}1/|W^s_{j, n}|<\rho_k$.

Let $(y_l)_{l\ge 1}\subset c_{00}^N$ dense in $c^N_0(\mathbb{N})$ such that for every $l\ge 1$, every $s\in [N]$, $\deg(y^s_l)\le l$, $\|y^s_l\|\le l$ and for every $j\le l$, $y^s_{l,j}\ne 0$. We can then assume, without loss of generality, that the sequence $(l\rho_l)_l$ tends to $0$, that 
for every $s\ne t\in[N]$, any $n\in A_l$, any $j\in f^n_t([l])\cap f_s^n([l])$, 
\[
\left|\frac{W^s_{f_s^{-n}(j),n}}{W^t_{f_t^{-n}(j),n}}-\frac{y^s_{l,f_s^{-n}(j)}}{y^t_{l,f_t^{-n}(j)}}\right|\le \frac{1}{l^2}
\]
and that for any $j\in f^{n}_t([l])\cap f_s^{n}((l, +\infty))$,
\[
\frac{\abs{W^s_{f_s^{-n}(j),n}}}{\abs{W^t_{f_t^{-n}(j),n}}} \le \frac{1}{l^2}.
\]

We can now show that $T_1,\dots,T_N$ satisfy the Disjoint $\mathcal{A}$-Hypercyclicity Criterion for $(y_l)$, $(A_l)$, $\varepsilon_l=\max\{l\rho_l,\frac{2}{l}\}$ and
\[S_n(y_l)=\sum_{t=1}^{N}\sum_{j\in J_{n,t,l}}\frac{1}{W^t_{j,n}}y^{t}_{l, j}e_{f_t^n(j)} \]
where $J_{n,1,l}=[l]$ and $J_{n,t+1,l}=[l]\backslash f_{t+1}^{-n}(\bigcup_{s\le t}f_{s}^n([l]))$ for every $t\in [N-1]$.\\

Again we note that our assumptions on $(A_k)$ implies that if $m\in A_l$ and $n\in A_k$, then for every $t,s\in [N]$, if $n\ne m$ or $s\ne t$, we have
 \[f^m_t(J_{m,t,l})\cap f^n_{s}(J_{n,s,k})=\emptyset.\]

\begin{itemize}
 \item Let $\varepsilon>0$ and $k\ge 1$ such that $j\rho_j\le \varepsilon$ for every $j\ge k$. We have for every finite set $F\subset [N_k,+\infty)$ and every $l\le k$,
\begin{align*}
\|\sum_{n\in A_l\cap F}S_n y_l\|&=\max_{n\in A_l\cap F}\|S_n y_l\|\\
&=\max_{n\in A_l\cap F,\ t\in[N],\ j\in J_{n,t,l}}\frac{1}{|W^t_{j,n}|}|y^{t}_{l,j}|\\
&\le (\max_t\|y^{t}_{l}\|)\max_{n\in A_l\cap [N_k,+\infty), t\in[N], j\in [l]}\frac{1}{|W^t_{j,n}|}\le l\rho_k\le \varepsilon
\end{align*}
and for every $l>k$,
\begin{align*}
\|\sum_{n\in A_l\cap F}S_n y_l\|
&=\max_{n\in A_l\cap F, t\in [N], j\in J_{n,t,l}}\frac{1}{|W^t_{j,n}|}|y^{t}_{l,j}|\\
&\le l\max_{n\in A_l, t\in [N], j\in [l]}\frac{1}{|W^t_{j,n}|}\le 
l\rho_l\le \varepsilon.
\end{align*}
We conclude that $\sum_{n\in A_l}S_n y_l$ converges unconditionally in $c_0(\mathbb{N})$ uniformly in $l$.
\item Let $k,l\ge 1$ and $s\in [N]$. For any $n\in A_k$, we have
\begin{align*}
\Big\|\sum_{i\in A_l\backslash\{n\}}T_s^nS_iy_l\Big\|
&=\Big\|\sum_{i\in A_l\backslash\{n\}}\sum_{t=1}^{N}\sum_{j\in J_{i,t,l}}\frac{1}{W^t_{j,i}}y^{t}_{l,j}T_s^ne_{f_t^i(j)}\Big\|\\ 
&\le (\max_t\|y^{t}_{l}\|)\max_{t\in [N], i\in A_l\backslash\{n\}, j\in J_{i,t,l}\ \text{with}\ f_t^i(j)\in f_s^n([1,+\infty))}\frac{|W^s_{f_s^{-n}(f_t^i(j)),n}|}{|W^t_{j,i}|}\\
&\le  l\max_{t\in [N], i\in A_l\backslash\{n\}, j\in [l]\ \text{with}\ f_t^i(j)\in f_s^n((k,+\infty))}\frac{|W^s_{f_s^{-n}(f_t^i(j)),n}|}{|W^t_{j,i}|}\\
&\le  l\rho_l \le \varepsilon_l,
\end{align*}
where the case $s=t$ follows from (1).
\item Let $l\ge 1$ and $s\in [N]$. For any $n\in A_l$, any $k<l$, we have
\begin{align*}
\Big\|\sum_{i\in A_k}T_s^nS_iy_k\Big\|&=\Big\|\sum_{i\in A_k}\sum_{t=1}^{N}\sum_{j\in J_{i,t,k}}\frac{1}{W^t_{j,i}}y^{t}_{k,j}T_s^ne_{f_t^i(j)}\Big\| \\
&\le (\max_t\|y^{t}_{l}\|)\max_{t\in [N], i\in A_k, j\in [k]:f_t^i(j)\in f_s^n([1,+\infty))}\frac{|W^s_{f_s^{-n}(f_t^i(j)),n}|}{|W^t_{j,i}|}\\
&\le  l\max_{t\in [N],j\in [k], i\in A_k:f_t^i(j)\in f_s^n((l,+\infty))}\frac{|W^s_{f_s^{-n}(f_t^i(j)),n}|}{|W^t_{j,i}|}\\
&\le l\rho_l\le \varepsilon_l,
\end{align*}
where as previously, the case $s=t$ follows from (1).

\item Let $l\ge 1$ and $s\in [N]$. For any $n\in A_l$, we get, by definition of $J_{n,t,l}$,
\begin{align*}
\|T_s^nS_ny_l-y^s_{l}\|&=
\|\sum_{t=1}^{N}\sum_{j\in J_{n,t,l}}\frac{1}{W^t_{j,n}}y^{t}_{l,j}T_s^ne_{f_t^n(j)}-\sum_{j\in [l]}y_{l,j}^se_j\|\\
&=\|\sum_{t=1}^{N}\sum_{j\in J_{n,t,l}: f_t^n(j)\in f_s^n([1,+\infty))}\frac{W^s_{f^{-n}_s(f_t^n(j)),n}}{W^t_{j,n}}y^{t}_{l,j}e_{f^{-n}_s(f_t^n(j))}-\sum_{j\in [l]}y_{l,j}^se_j\|\\
&=\|\sum_{t=1}^{s}\sum_{j\in J_{n,t,l}: f_t^n(j)\in f_s^n([l])}\frac{W^s_{f^{-n}_s(f_t^n(j)),n}}{W^t_{j,n}}y^{t}_{l,j}e_{f^{-n}_s(f_t^n(j))}-\sum_{j\in [l]}y_{l,j}^se_j\|\\
&\quad+\max_{\substack{t\in [N]\backslash\{s\},\\ j\in J_{n,t,l}: f_t^n(j)\in f_s^n((l,+\infty))}}\frac{|W^s_{f^{-n}_s(f_t^n(j)),n}|}{|W^t_{j,n}|}|y^{t}_{l,j}|.
\end{align*}

If $i\in [l]$ then there are unique $t_i\le s$ and $j_i\in J_{n,t_i,l}$ such that
$f_s^{-n}(f^n_{t_i}(j_i))=i$. It follows that
\begin{align*}
&\|T_s^nS_ny_l-y^s_{l}\|\\
 &\quad\le \left\|\sum_{i\in [l]}
\Big(\frac{W^s_{i,n}}{W^{t_i}_{j_i,n}}y^{t_i}_{l,j_i}e_{i}-y_{l,i}^se_i\Big)\right\|+(\max_{t\in [N]}\|y_l^t\|)\max_{\substack{t\in [N]\backslash\{s\},\\ i\in f_s^n((l,+\infty)) \cap f_t^n([l])}}\frac{|W^s_{f^{-n}_s(i),n}|}{|W^t_{f^{-n}_t(i),n}|}\\
&\quad\le \max_{i\in [l]}
\left|\frac{W^s_{i,n}}{W^{t_i}_{j_i,n}}y^{t_i}_{l,j_i}-y_{l,i}^s\right|^p + l\max_{\substack{t\in [N]\backslash\{s\},\\ i\in f_s^n((l,+\infty)) \cap f_t^n([l])}}\frac{|W^s_{f^{-n}_s(i),n}|}{|W^t_{f^{-n}_t(i),n}|}\\
&\quad\le l\max_{i\in [l]} \left|\frac{W^s_{i,n}}{W^{t_i}_{j_i,n}}-\frac{y_{l,i}^s}{y^{t_i}_{l,j_i}}\right|+\frac{1}{l}\\
&\quad\le l\max_{t\le s, k\in f^n_t([l])\cap f^n_s([l])} \left|\frac{W^s_{f_s^{-n}(k),n}}{W^{t}_{f_t^{-n}(k),n}}-\frac{y_{l,f_s^{-n}(k)}^s}{y^{t}_{l,f_t^{-n}(k)}}\right|+\frac{1}{l}\\
&\quad\le \frac{2}{l}.
\end{align*}
\end{itemize}
\end{proof}

%REMARK
\begin{remark}\label{remreg}
We remark that the condition (1) in the previous theorems gives us a characterization of $\mathcal{A}$-hypercyclic pseudo-shifts. The obtained statement is comparable to the characterization given for weighted shifts by Grosse-Erdmann in \cite{Gr}. It is also important to remark that the assumption of weakly partition regularity is only used in the proof of the second implication.
\end{remark}

The proofs of the previous theorems show that pseudo-shifts are disjoint $\mathcal{A}$-hypercyclic if and only if they satisfy the Disjoint $\mathcal{A}$-Hypercyclicity Criterion.

\begin{corollary}
Let $T_{f_1,w_1},\dots,T_{f_N,w_N}$ be unilateral pseudo-shifts on $c_0(\mathbb{N})$ or $\ell^p(\mathbb{N})$ with $1\le p<\infty$ and $\mathcal{A} = \A_{uBd}$, $\A_{ud}$ or $\mathcal{A}_{ld}$. If $\A$ is $(f_s)_{s=1}^N$-weakly partition regular then the operators $T_{f_1,w_1},\dots,T_{f_N,w_N}$ are disjoint $\mathcal{A}$-hypercyclic if and only if they satisfy the Disjoint $\mathcal{A}$-Hypercyclicity Criterion.
\end{corollary} 

Thanks to Theorem \ref{caraclp} and Theorem \ref{caracc0}, we also get a characterization of disjoint $\mathcal{A}$-hypercyclicity for any family of weighted shifts. Indeed, for weighted shifts, we have $f^m_s([l])=[m+1, m+l]$ for any $s\in [N]$, then it follows from Lemma~\ref{lemma.polansky2} that $\A_{\infty}$, $\A_{uBd}$, $\A_{ud}$ and $\mathcal{A}_{ld}$ are $(f_s)_{s=1}^N$-weakly partition regular. We state this characterization in the case of $\ell^p(\N)$.

%THEOREM

\begin{theorem}[Characterization for weighted shifts on $\ell^p(\N)$]\label{shiftlp}
Let $T_1=B_{w_1},\dots,T_N=B_{w_N}$ be unilateral weighted shifts on $\ell^p(\mathbb{N})$ with $1\le p<\infty$ and $\mathcal{A} = \A_{uBd}$, $\A_{ud}$ or $\mathcal{A}_{ld}$. The family $T_1,\dots,T_n$ are disjoint $\mathcal{A}$-hypercyclic if and only if
 there exist disjoint sets $A_k\in \A$, $k\ge 1$, and a sequence $(\varepsilon_k)$ of positive real numbers tending to $0$ such that
 \begin{enumerate}
\item for any $s\in [N]$, we have  $\sum_{n\ge 1}\frac{1}{|W^s_{1, n}|^p}<\infty$,

%\item for any $l, k \geq 1$ with $n\in A_l, m\in A_k$,
%for any $t, s\in [N]$ with $t\neq s$, and any $j\in f^m_t([k])\cap f_s^n((l, \infty))$
%\[
%\frac{\abs{W^s_{f_s^{-n}(j),n}}}{\abs{W^t_{f_t^{-m}(j),m}}}\le \min\{\varepsilon_l, \varepsilon_k\},
%\]{

\item for any $1\le j\le l$, any $k\ge 1$, any $n\in A_k$, any $s\ne t\in[N]$,
\[\sum_{m\in A_l: m>n} \frac{\abs{W^s_{j+m-n,n}}^p}{\abs{W^t_{j,m}}^p}\le \varepsilon_l,\]
\item for any $1\le j\le k<l$, any $n\in A_l$, any $s\ne t\in [N]$,
\[\sum_{m\in A_k:m> n} \frac{\abs{W^s_{j+m-n,n}}^p}{\abs{W^t_{j,m}}^p}\le \varepsilon_l,\]
\item for any $l$, any $(a_{s,j})_{s\in[N],j\in [l]}$ with $a_{s,j}\ne 0$, any $\varepsilon>0$, there exists $L\ge l$ such that for every $s\ne t\in[N]$, every $n\in A_L$, every $j\in[l]$, 
\[
\left|\frac{W^s_{j,n}}{W^t_{j,n}}-\frac{a_{s,j}}{a_{t,j}}\right|<\varepsilon
\]
\end{enumerate}
where $W^s_{l, n}=\prod_{\nu=1}^nw_{s, l+\nu}$, for all $s\in [N]$.
\end{theorem}
\begin{proof}
It suffices to remark that by Lemma~\ref{lemma.polansky2}, we can always assume that for every $n\in A_l$, $m\in A_{k}$ with $n\ne m$, $|n-m|\ge l+k+2$. The characterization for weighted shifts then follows from Theorem~\ref{caraclp}
\end{proof}

We can also apply Theorem~\ref{caraclp} and Theorem~\ref{caracc0} to powers of weighted shifts. Indeed, in this case, we have $f^m_s([l])=[q_s m+1, q_s m+l]$ for any $s\in [N]$ and some $q_s\in \mathbb{N}$, and it follows from Lemma~\ref{lemma.polansky2} that $\A_{\infty}$, $\A_{uBd}$, $\A_{ud}$ and $\mathcal{A}_{ld}$ are $(f_s)_{s=1}^N$-weakly partition regular. Here is an example of applications.

%COROLLARY

\begin{corollary}
Let $T_1=\lambda_1 B^{i_1}$,...,$T_N=\lambda_N B^{i_N}$ with $1<|\lambda_1|\le \cdots\le |\lambda_N|$ and $i_1,\dots, i_N\ge 1$. Let $X=\ell^p(\mathbb{N})$ with $1\le p<\infty$ or $X=c_0(\mathbb{N})$. Then the following assertions are equivalent:
\begin{enumerate}
\item $T_1$,...,$T_N$ are disjoint frequently hypercyclic on $X$;
\item $T_1$,...,$T_N$ are disjoint upper frequently hypercyclic on $X$;
\item $T_1$,...,$T_N$ are disjoint reiteratively hypercyclic on $X$;
\item $T_1$,...,$T_N$ are disjoint hypercyclic on $X$;
\item for every $1\le s<t\le N$, $i_s< i_t$ and $|\lambda_s|<|\lambda_t|$.
\end{enumerate}
\end{corollary}
\begin{proof}
It is obvious that $(1)\Rightarrow (2)\Rightarrow (3)\Rightarrow (4)$. Moreover, $(4) \Rightarrow (5)$ is a consequence of \cite{BePe07}. It remains to show that $(5)\Rightarrow (1)$. We only write the proof for $\ell^p(\mathbb{N})$, as the proof for $c_0(\mathbb{N})$ is similar. In view of Theorem~\ref{caraclp}, we need to find a sequence $(A_k)_{k\ge 1}\subset \A_{ld}$ and a sequence $(\varepsilon_k)_{k\ge 1}$ of positive real numbers tending to $0$ such that
 \begin{enumerate}
\item for any $1\le j\le l$, any $k\ge 1$, any $n\in A_k$, any $s\ne t\in[N]$
\[\sum_{m\in A_l\backslash\{n\}: j+i_tm> k+i_sn} \frac{\abs{\lambda_s}^{np}}{\abs{\lambda_t}^{mp}}\le \varepsilon_l,\]
\item for any $1\le j\le k<l$, any $n\in A_l$, any $s\ne t\in [N]$
\[\sum_{m\in A_k:j+i_tm> l+i_s n} \frac{\abs{\lambda_s}^{np}}{\abs{\lambda_t}^{mp}}\le \varepsilon_l,\]
\item for any $l$, any $(a_{s,j})_{s\in[N],j\in [l]}$ with $a_{s,j}\ne 0$, any $\varepsilon>0$, there exists $L\ge l$ such that for any $s\ne t\in[N]$, any $n\in A_L$, 
\begin{enumerate}
\item for any $j\in [1+i_tn,l+i_tn]\cap [1+i_sn,l+i_sn]$, 
\[
\left|\frac{\lambda_s^n}{\lambda_t^n}-\frac{a_{s,f_s^{-n}(j)}}{a_{t,f_t^{-n}(j)}}\right|<\varepsilon
\]
\item if  $[1+i_tn,l+i_tn]\cap (l+i_sn, +\infty)\ne \emptyset$ then
\[
\frac{\abs{\lambda_s}^n}{\abs{\lambda_t}^n}<\varepsilon.
\]
\end{enumerate}
\end{enumerate}

Let $\delta=\frac{i_1}{i_N}$, $\gamma=\max_{s<N}\frac{\log(|\lambda_s|)}{\log(|\lambda_{s+1}|)}$, $\Gamma=\frac{\log(|\lambda_{N}|)}{\log(|\lambda_{1}|)}$ so that for every $1\le s<t\le N$ and every $n\ge 0$, $|\lambda_s|^{n}\le |\lambda_t|^{\gamma n}$ and $|\lambda_t|^n\le |\lambda_s|^{\Gamma n}$. We consider  $1<r<\min\{\frac{1}{\gamma},\min_{s<N}\frac{i_{s+1}}{i_s}\}$,  $M>\max\{\frac{2}{\delta},\Gamma\}$
and $B=\mathbb{N}\cap \bigcup_{j\ge 1}(M^j r^j, M^j r^{j+1}-j)\in \mathcal{A}_{ld}$. Finally, we select, thanks to Lemma~\ref{lemma.polansky2}, a family $(A_k)_{k\ge 1}$ of sets in  $\mathcal{A}_{ld}$ such that for every $k,l$, $A_k\subset B$, $\min(A_k)\ge M^kr^{k}+Mk$ and for every $n\in A_k$, $m\in A_{l}$ with $n\ne m$, $|n-m|\ge k+l+2$.

 \begin{enumerate}
\item We first remark that if $m\in A_l$ and $m\ge \delta n-l$ with $n\in B$ then $m\ge \gamma (n+l)$.
Indeed, if $n\in (M^{l'} r^{l'}, M^{l'} r^{l'+1}-l')$ with $l'< l$ then since $m\ge \min(A_l)\ge M^lr^{l}+l$, we have $m\ge M^{l'}r^{l'+1}+l\ge n+l\ge \gamma(n+l)$ and if $l'\ge l$ then 
\[2m\ge m+l \ge \delta n> \delta M^{l'}r^{l'}\]
and thus
\[ m\ge \frac{\delta M}{2} M^{l'-1}r^{l'}\ge M^{l'-1}r^{l'} \]
Finally, since $m\in B$, we get $m\ge M^{l'} r^{l'}$ and thus $m\ge \frac{1}{r}(n+l')\ge \gamma (n+l')\ge \gamma(n+l)$. \\

On the other hand, if $m\in A_l$ and $m> rn-l$ with $n\in B$ then $m\ge  \Gamma(n+l)$. Indeed, 
if $n\in (M^{l'} r^{l'}, M^{l'} r^{l'+1}-l')$ with $l'< l$ then
$m\ge M^lr^l+Ml\ge M(n+l)\ge \Gamma(n+l)$
and if $l'\ge l$ then 
$m> rn-l\ge M^{l'}r^{l'+1}-l'$ and since $m\in B$, we get $m\ge M^{l'+1}r^{l'+1}\ge M(n+l)\ge  \Gamma(n+l)$.\\

Let $1\le j\le l$, $k\ge 1$, $n\in A_k$ and $s\ne t\in[N]$, if $s<t$, we then get
\begin{align*}
\sum_{m\in A_l\backslash\{n\}: j+i_tm>k+i_sn} \frac{\abs{\lambda_s}^{np}}{\abs{\lambda_t}^{mp}}&\le 
\sum_{m\in A_l: m> \frac{i_sn-l}{i_t}} \frac{\abs{\lambda_s}^{np}}{\abs{\lambda_t}^{mp}}\\
&\le 
\sum_{m\in A_l: m> \delta n-l} \frac{\abs{\lambda_s}^{np}}{\abs{\lambda_t}^{mp}}\\
&\le \sum_{m\ge \gamma (n+l)}\frac{\abs{\lambda_s}^{np}}{\abs{\lambda_t}^{mp}}\\
&\le \sum_{m\ge \gamma (n+l)}\frac{1}{\abs{\lambda_t}^{(m-\gamma n)p}}\\
&\le \sum_{m\ge 0}\frac{1}{\abs{\lambda_1}^{(m+\gamma l)p}}.
\end{align*}
and if $s>t$, we get
\begin{align*}
\sum_{m\in A_l\backslash\{n\}: j+i_tm>k+i_sn} \frac{\abs{\lambda_s}^{np}}{\abs{\lambda_t}^{mp}}&\le 
\sum_{m\in A_l: m> \frac{i_s}{i_t}n-l} \frac{\abs{\lambda_s}^{np}}{\abs{\lambda_t}^{mp}}\\
&\le 
\sum_{m\in A_l: m> rn-l} \frac{\abs{\lambda_s}^{np}}{\abs{\lambda_t}^{mp}}\\
&\le 
\sum_{m\ge \Gamma(n+l)} \frac{\abs{\lambda_s}^{np}}{\abs{\lambda_t}^{mp}}\\
&\le 
\sum_{m\ge \Gamma(n+l)} \frac{1}{\abs{\lambda_t}^{(m-\Gamma n)p}}\\
&\le \sum_{m\ge 0} \frac{1}{\abs{\lambda_1}^{(m+\Gamma l)p}}
\end{align*}
\item Let $1\le j\le k<l$, $n\in A_l$ and $s\ne t\in [N]$.
If $s<t$, it suffices to remark that if $m>\delta n-l$ and $m\in B$ then since $n\in A_l$, we have $m>\delta Ml -l\ge l$ and $n\in (M^{l'}r^{l'},M^{l'}r^{l'+1}-l')$ for some $l'\ge l$. It follows as previously that $m\ge \gamma(n+l)$ and thus
\begin{align*}
\sum_{m\in A_k:j+i_tm>l+i_s n} \frac{\abs{\lambda_s}^{np}}{\abs{\lambda_t}^{mp}}&\le 
\sum_{m\ge \gamma(n+l)} \frac{\abs{\lambda_s}^{np}}{\abs{\lambda_t}^{mp}}\le \sum_{m\ge 0} \frac{1}{\abs{\lambda_1}^{(m+\gamma l)p}}.
\end{align*}
On the other hand, if $s>t$ then  if $m\in A_k$ and $m> rn-l$ , we have $m\ge  \Gamma(n+l)$. Indeed, 
since $n\in A_l$, $n\in (M^{l'} r^{l'}, M^{l'} r^{l'+1}-l')$ with $l'\ge l$ and thus
$m>rn-l\ge M^{l'}r^{l'+1}-l'$ and since $m\in B$, we get $m\ge M^{l'+1}r^{l'+1}\ge M(n+l)\ge  \Gamma(n+l)$. It follows that
\begin{align*}
\sum_{m\in A_k:j+i_tm>l+i_s n} \frac{\abs{\lambda_s}^{np}}{\abs{\lambda_t}^{mp}}&\le 
\sum_{m\ge \Gamma(n+l)} \frac{\abs{\lambda_s}^{np}}{\abs{\lambda_t}^{mp}}\le \sum_{m\ge 0} \frac{1}{\abs{\lambda_1}^{(m+\Gamma l)p}}.
\end{align*}
We can thus consider $\varepsilon_l=\sum_{m\ge 0} \frac{1}{\abs{\lambda_1}^{(m+\gamma l)p}}$.
\item Let $l\ge 1$, $(a_{s,j})_{s\in[N],j\in [l]}$ with $a_{s,j}\ne 0$ and $\varepsilon>0$. If we consider $L\ge l$ such that $\frac{\abs{\lambda_s}^L}{\abs{\lambda_{s+1}}^L}<\varepsilon$ for every $s<N$, we get that for every $s\ne t\in[N]$, any $n\in A_L$, 
\begin{enumerate}
\item the sets $[1+i_tn,l+i_tn]$ and $[1+i_sn,l+i_sn]$ are disjoint because if $t>s$, we have
\[1+i_tn\ge 1+i_sn+n > l+i_sn\]
since $\min A_L \ge l$.

\item if $s>t$ then the sets $[1+i_tn,l+i_tn]$ and $(l+i_sn, \infty)$ are disjoint and if $s<t$, we have
\[
\frac{\abs{\lambda_s}^n}{\abs{\lambda_t}^n}\le \frac{\abs{\lambda_s}^L}{\abs{\lambda_t}^L}<\varepsilon.
\]
\end{enumerate}
\end{enumerate}
\end{proof}

% NEW SECTION
\section{Disjoint reiterative and upper frequent hypercyclicity cases}

In \cite{BoGr}, Bonilla and Grosse-Erdmann gave a Birkhoff-type characterization for upper frequent and reiterative hypercyclicity. They also used this result to simplify the characterizations of upper frequently and reiteratively hypercyclic weighted shifts that appeared in the literature. In view of these results, we can also simplify our characterization of disjoint $\mathcal{A}$-hypercyclic pseudo-shifts in the case of $\mathcal{A}_{uBd}$ and $\mathcal{A}_{ud}$ and give this characterization for a larger family of pseudo-shifts. In order to do this, we introduce the following notion.

\begin{definition}
Let $f_s:\N\to \N,\ s\in[N]$ be increasing maps and $\A$ be a family of subsets of $\N$. The family $\A$ is called \emph{$(f_s)_{s=1}^N$-almost partition regular} if for any $l\geq 1$, any $B\in \A$ there exists $A\subseteq B, A\in \A$ such that 
\[
f_s^m([l])\cap f_t^n([l])=\emptyset  \qquad \forall m, n\in A, m\neq n, \forall s,t\in [N].
\]
\end{definition}

\begin{theorem}[Bonilla and Grosse-Erdmann \cite{BoGr}]
Let $X$ be a complete metric space, $Y$ a separable metric space and $T_n:X\to Y$, $n\ge 0$, continuous mappings. Let $\mathcal{A}=\bigcup_{\delta\in D}\bigcap_{\mu\in M} \mathcal{A}_{\delta,\mu}$ be an upper Furstenberg family. If for any non-empty open subset $V$ of $Y$, there is some $\delta\in D$ such that for any non-empty open subset $U$ of $X$, there is some $x\in U$ such that 
\[\{n\ge 0:T_nx\in V\}\in \bigcap_{\mu\in M} \mathcal{A}_{\delta,\mu}\]
then $(T_n)$ admits an $\mathcal{A}$-universal point, i.e. there exists $x\in X$ such that for any non-empty open set $U$ in $Y$, $\{n\ge 0:T_nx\in U\}\in \mathcal{A}$.
\end{theorem}

By considering the sequence of operators $(T^n_1\times\cdots\times T^n_N)\in L(X,X^N)$, we get the following.

%THEOREM

\begin{theorem}
Let $X$ be a separable Banach space, $T_1,\cdots,T_N\in L(X)$. If for any non-empty open subsets $V_1,\dots, V_N$ of $X$, there is some $\delta>0$ such that for any non-empty open subset $U$ of $X$, there is some $x\in U$ such that 
\[\overline{\emph{dens}}(\bigcap_{s\in [N]}\{n\ge 0:T^n_s x\in V_s\})>\delta \quad\text{\emph{(}resp. $\overline{\emph{Bd}}(\bigcap_{s\in [N]}\{n\ge 0:T^n_sx\in V_s\})>\delta$\emph{)}}\]
then $T_1,\cdots,T_N$ are disjoint upper frequently hypercyclic (resp.  disjoint reiteratively hypercyclic).
\end{theorem}

In the case of pseudo-shifts, we can still simplify this criterion since pseudo-shifts have dense generalized kernel.

%THEOREM

\begin{theorem}\label{thmsimp}
Let $T_1=T_{f_1,w_1},\dots,T_N=T_{f_N,w_N}$ be unilateral pseudo-shifts on $X=\ell^p(\mathbb{N})$ with $1\le p<\infty$ or $c_0(\mathbb{N})$. If for any non-empty open subsets $V_1,\dots, V_N$ of $X$,  there is some $x\in X$ such that
\[\overline{\emph{dens}}(\bigcap_{s\in [N]}\{n\ge 0:T^n_s x\in V_s\})>0 \quad\text{\emph{(}resp. $\overline{\emph{Bd}}(\bigcap_{s\in [N]}\{n\ge 0:T^n_s x\in V_s\})>0$\emph{)}}\]
then $T_1,\cdots,T_N$ are disjoint upper frequently hypercyclic (resp. disjoint reiteratively hypercyclic).
\end{theorem}
\begin{proof}
Let  $V_1,\dots, V_N$ be non-empty open subsets of $X$ and $x\in X$ such that
\[\delta:=\overline{\text{dens}}(\bigcap_{s\in [N]}\{n\ge 0:T^n_s x\in V_s\})>0 \quad\text{(resp. $\delta:=\overline{Bd}(\bigcap_{s\in [N]}\{n\ge 0:T^n_s x\in V_s\})>0$)}.\]
Let $U$ be a non-empty open subset of $X$. It suffices to remark that there exists $y\in c_{00}$ such that $x+y\in U$ and that
\[\overline{\text{dens}}(\bigcap_{s\in [N]}\{n\ge 0:T^n_s(x+y)\in V_s\})=\delta \quad\text{(resp. $\delta:=\overline{Bd}(\bigcap_{s\in [N]}\{n\ge 0:T^n_s(x+y)\in V_s\})=\delta$)}\]
since $T^n_s(y)=0$ for $n$ sufficiently big.
\end{proof}

We are now in a good position to give a simplified characterization of disjoint reiterative and upper frequent hypercyclicity for pseudo-shifts. We first focus on the $\ell^p(\mathbb{N})$ case.
%THEOREM

\begin{theorem}
\label{simpcaraclp}
Let $T_1=T_{f_1,w_1},\dots,T_N=T_{f_N,w_N}$ be unilateral pseudo-shifts on $\ell^p(\mathbb{N})$ with $1\le p<\infty$ and $\mathcal{A} = \A_{uBd}$ or $\A_{ud}$. If $\A$ is $(f_s)_{s=1}^N$-almost partition regular then the family $T_1,\dots,T_n$ are disjoint $\mathcal{A}$-hypercyclic if and only if
\begin{enumerate}
\item for any $j\ge 1$ and any $s\in [N]$, we have  $\sum_{n\ge 1}\frac{1}{|W^s_{j, n}|^p}<\infty$,

\item  for any $l\ge 1$, any $(a_{s,j})_{s\in[N],j\in [l]}$ with $a_{s,j}\ne 0$, any $\varepsilon>0$, there exists $A\in \mathcal{A}$ such that for any $s\ne t\in [N]$, any $n\in A$,
\begin{enumerate}
\item 
for any $1\le j\le l$,
\[\sum_{m\in A: f_t^m(j)\in f_s^n((l,+\infty))} \frac{\abs{W^s_{f_s^{-n}(f_t^m(j)),n}}^p}{\abs{W^t_{j,m}}^p}\le \varepsilon,\]
\item for any $j\in f^n_t([l])\cap f_s^n([l])$,
\[
\left|\frac{W^s_{f_s^{-n}(j),n}}{W^t_{f_t^{-n}(j),n}}-\frac{a_{s,f_s^{-n}(j)}}{a_{t,f_t^{-n}(j)}}\right|\le \varepsilon
\]
\end{enumerate}
\end{enumerate}
where $W^s_{l, n}=\prod_{\nu=1}^nw_{s, f_s^\nu(l)}$, for all $s\in [N]$.
\end{theorem}
\begin{proof}
If $T_1,\dots,T_N$ are disjoint $\mathcal{A}$-hypercyclic, we can apply Theorem~\ref{caraclp} in view of Remark~\ref{remreg}. It suffices then to remark that the condition (2)(a) of Theorem~\ref{simpcaraclp} follows from the condition (2) of Theorem~\ref{caraclp} for $m\ne n$ and from the condition (4)(b) of Theorem~\ref{caraclp} for $m=n$.

%On the other hand, if we consider $n\in A_l, m\in A_k$ with $m>n$, $s\in [N]$ and $j\in f^m_s([k])\cap f_s^n((l, \infty))$, we have $\|T_s^nx-a^l_s\|<\tau_l$ and thus $|W^s_{f_s^{-n}(j),n}x_j-a^l_{s,f_s^{-n}(j)}|=|W^s_{f_s^{-n}(j),n}x_j|<\tau_l$ since $f_s^{-n}(j)>l$. On the other hand, we have $\|T_s^mx-a^k_s\|<\tau_k$ and thus $|W^s_{f_s^{-m}(j),m}x_j-a^k_{s,f_s^{-m}(j)}|<\tau_k$. It follows that
%$|x_j|>0$ and that $|W^s_{f_s^{-m}(j),m} x_j|>\rho_k-\tau_k$ and thus
%\[
%\frac{\abs{W^s_{f_s^{-n}(j),n}}}{\abs{W^s_{f_s^{-m}(j),m}}}<\frac{\tau_l}{\rho_k-\tau_k}
%\]

We now prove the other implication thanks to Theorem~\ref{thmsimp}. Let $V_1,\dots, V_N$ be non-empty open subsets of $\ell^p(\mathbb{N})$. There exist $l\ge 1$, $(a_{s,j})_{s\in[N],j\in [l]}$ with $a_{s,j}\ne 0$ and $\varepsilon>0$ such that for every $s\in [N]$, $B(a_s,\varepsilon)\subset V_s$ where $a_s=\sum_{j=1}^l a_{s,j}e_j$.

By assumption, we know that  for any $j\ge 1$ and any $s\in [N]$, we have  \[\sum_{n\ge 1}\frac{1}{|W^s_{j, n}|^p}<\infty\]
and that for every $\tau>0$, there exists $A\in \mathcal{A}$ such that for any $s\ne t\in [N]$, any $n\in A$,
\begin{enumerate}
\item 
for any $1\le j\le l$,
\[\sum_{m\in A: f_t^m(j)\in f_s^n((l,+\infty))} \frac{\abs{W^s_{f_s^{-n}(f_t^m(j)),n}}^p}{\abs{W^t_{j,m}}^p}\le \tau,\]
\item for any $j\in f^n_t([l])\cap f_s^n([l])$,
\[
\left|\frac{W^s_{f_s^{-n}(j),n}}{W^t_{f_t^{-n}(j),n}}-\frac{a_{s,f_s^{-n}(j)}}{a_{t,f_t^{-n}(j)}}\right|\le \tau.
\]
\end{enumerate}
Let $M\ge 1$ such that for any $j\le l$ and any $s\in [N]$, we have  \[\sum_{n\ge M}\frac{1}{|W^s_{j, n}|^p}\le \tau.\]
Since $\mathcal{A}$ is $(f_s)_{s=1}^N$-almost partition regular, we can assume that
\[
f_s^m([l])\cap f_t^n([l])=\emptyset  \qquad \forall m, n\in A, m\neq n, \forall s, t\in [N]
\]
and that for any $m\ne n\in A$, $\abs{m-n}\ge M$. 

We then let
\[x=\sum_{n\in A}\sum_{t=1}^{N}\sum_{j\in J_{n,t}}\frac{1}{W^t_{j,n}}a_{t,j}e_{f_t^n(j)}\]
where $J_{n,1}=[l]$ and $J_{n,t+1}=[l]\backslash f_{t+1}^{-n}(\bigcup_{s\le t}f_{s}^n([l]))$ for every $t\in [N-1]$.  
By the definition of sets $J_{n,s}$ and our assumptions on $A$, we get that for every $m\in A$, $n\in A$, if $m\ne n$ or $t\ne s\in [N]$ then 
 \[f^m_t(J_{m,t})\cap f^n_{s}(J_{n,s})=\emptyset.\]
It follows that the vector $x$ is well-defined since
\begin{align*}
\|x\|^p
&=\sum_{n\in A}\sum_{t=1}^{N}\sum_{j\in J_{n,t}}\frac{1}{|W^t_{j,n}|^p}|a_{t,j}|^p\\
&\le Nl(\max_{s\in [N]}\|a_s\|^p)\max_{j\in [l], t\in [N]}\left(\sum_{n\in A}\frac{1}{|W^t_{j,n}|^p}\right)<\infty.
\end{align*}
It remains to show that  for every $n\in A$, every $s\in [N]$, $\|T^n_sx-a_s\|<\varepsilon$.
Let $n\in A$ and $s\in [N]$. Note that for every $m\ne n$,
$f_s^{-n}(f_t^m([l]))\cap [l]=\emptyset$ and that if $i\in [l]$ then there are unique $t_i\in [N]$ and $j_i\in J_{n,t_i}$ such that $f_s^{-n}(f^n_{t_i}(j_i))=i$. Therefore, we have 
\begin{align*}
\|T^n_sx-a_s\|^p&=\|\sum_{m\in A}\sum_{t=1}^{N}\sum_{j\in J_{m,t}}\frac{1}{W^t_{j,m}}a_{t,j}T^n_se_{f_t^m(j)}-a_s\|^p\\
&\le \left\|\sum_{i=1}^l(\frac{W^s_{i,n}}{W^{t_i}_{j_i,n}}a_{t_i,j_i}e_i-a_{s,i}e_i)\right\|^p\\
&\quad + \|\sum_{m\in A}\sum_{t=1}^{N}\sum_{j\in J_{m,t}:f_t^m(j)\in f_s^n((l,+\infty))}a_{t, j}\frac{W^s_{f_s^{-n}(f_t^m(j)),n}}{W^t_{j,m}}e_{f_s^{-n}(f_t^m(j))}\|^p\\
&\le \sum_{i=1}^l|\frac{W^s_{i,n}}{W^{t_i}_{j_i,n}}a_{t_i,j_i}-a_{s,i}|^p\\
&\quad + Nl(\max_{s\in [N]}\|a_s\|^p)\max_{t\in [N],j\in [l]}\|\sum_{m\in A:f_t^m(j)\in f_s^n((l,+\infty))}\frac{W^s_{f_s^{-n}(f_t^m(j)),n}}{W^t_{j,m}}e_{f_s^{-n}(f_t^m(j))}\|^p\\
&\le \sum_{i=1}^l|\frac{W^s_{i,n}}{W^{t_i}_{j_i,n}}a_{t_i,j_i}-a_{s,i}|^p+ Nl(\max_{s\in [N]}\|a_s\|^p)\tau.
\end{align*}
The last inequality for $s=t$ follows from the fact that if $m\in A, j\in [l]$ and $f^m_s(j)\in f^n_s((l, \infty))$ then $m>n$ and $m-n\geq M$. Therefore, for any $j\le l$ and any $s\in [N]$, we have  \begin{align*}\|\sum_{m\in A:f_s^m(j)\in f_s^n((l,+\infty))}\frac{W^s_{f_s^{-n}(f_s^m(j)),n}}{W^s_{j,m}}e_{f_s^{-n}(f_s^m(j))}\|^p&=
\|\sum_{m\in A:f_s^m(j)\in f_s^n((l,+\infty))}\frac{1}{W^s_{j,m-n}}e_{f_s^{m-n}(j)}\|^p\\
&\le \sum_{n\ge M}\frac{1}{|W^s_{j, n}|^p}\le \tau.
\end{align*}

Finally, for any $i\in [l]$, if we let $k_i:=f_{t_i}^{n}(j_i)=f_s^{n}(i)$, we have
\begin{align*}
\Big|\frac{W^s_{i,n}}{W^{t_i}_{j_i,n}}a_{t_i,j_i}-a_{s,i}\Big|&=
\Big|\frac{W^s_{f_{s}^{-n}(k_i),n}}{W^{t_i}_{f_{t_i}^{-n}(k_i),n}}a_{t_i,f_{t_i}^{-n}(k_i)}-a_{s,f_{s}^{-n}(k_i)}\Big|\\
&\le (\max_{s\in [N]}\|a_s\|) \Big|\frac{W^s_{f_{s}^{-n}(k_i),n}}{W^{t_i}_{f_{t_i}^{-n}(k_i),n}}-\frac{a_{s,f_{s}^{-n}(k_i)}}{a_{t_i,f_{t_i}^{-n}(k_i)}}\Big| \le\max_{s\in [N]}\norm{a_s} \tau
\end{align*}

and we get the desired result by considering $\tau$ sufficiently small.
\end{proof}

We turn our attention now to the $c_0(\mathbb{N})$ case.
%THEOREM

\begin{theorem}
\label{simpcaracc0}
Let $T_1=T_{f_1,w_1},\dots,T_N=T_{f_N,w_N}$ be unilateral pseudo-shifts on $c_0(\mathbb{N})$ and $\mathcal{A} = \A_{uBd}$ or $\A_{ud}$. If $\A$ is $(f_s)_{s=1}^N$-almost partition regular then the family $T_1,\dots,T_n$ are disjoint $\mathcal{A}$-hypercyclic if and only if for any $l$, any $(a_{s,j})_{s\in[N],j\in [l]}$ with $a_{s,j}\ne 0$, any $\varepsilon>0$, there exists $A\in \mathcal{A}$ such that
\begin{enumerate}
\item for any $j\in [l]$ and any $s\in [N]$, we have  $\lim_{n\in A}|W^s_{j, n}|=\infty$,
\item 
for any $j\in [l]$, any $n\in A$, any $s, t\in[N]$, any $m\in A$ with $f_t^m(j)\in f_s^n((l,+\infty)),$
\[\frac{\abs{W^s_{f_s^{-n}(f_t^m(j)),n}}}{\abs{W^t_{j,m}}}\le \varepsilon,\]
\item for any $n\in A$, any $j\in f^n_t([l])\cap f_s^n([l])$ with $s\ne t$, 
\[
\left|\frac{W^s_{f_s^{-n}(j),n}}{W^t_{f_t^{-n}(j),n}}-\frac{a_{s,f_s^{-n}(j)}}{a_{t,f_t^{-n}(j)}}\right|\le \varepsilon
\]
\end{enumerate}
where $W^s_{l, n}=\prod_{\nu=1}^nw_{s, f_s^\nu(l)}$, for all $s\in [N]$.
\end{theorem}
\begin{proof}
If $T_1,\dots,T_N$ are disjoint $\mathcal{A}$-hypercyclic, we can apply Theorem~\ref{caracc0} in view of Remark~\ref{remreg}. It suffices then to remark that the condition (2) of Theorem~\ref{simpcaracc0} follows from the condition (2) of Theorem~\ref{caracc0} for $m\ne n$ and $s\ne t$, from  the condition (1) of Theorem~\ref{caracc0} for $m\ne n$ and $s=t$ and from the condition (4)(b) of Theorem~\ref{caracc0} for $m=n$ and $s\ne t$.

We now prove the other implication thanks to Theorem~\ref{thmsimp}. Let $V_1,\dots, V_N$ be non-empty open subsets of $c_0(\mathbb{N})$. There exist $l\ge 1$, $(a_{s,j})_{s\in[N],j\in [l]}$ with $a_{s,j}\ne 0$ and $\varepsilon>0$ such that for every $s\in [N]$, $B(a_s,\varepsilon)\subset V_s$ where $a_s=\sum_{j=1}^l a_{s,j}e_j$.

By assumption, we know that  for every $\tau>0$, there exists $A\in \mathcal{A}$ such that 
\begin{enumerate}
\item for any $j\in [l]$ and any $s\in [N]$, we have  $\lim_{n\in A}|W^s_{j, n}|=\infty$,
\item 
for any $j\in [l]$, any $n\in A$, any $s, t\in[N]$, any $m\in A$ with $f_t^m(j)\in f_s^n((l,+\infty))$,
\[\frac{\abs{W^s_{f_s^{-n}(f_t^m(j)),n}}}{\abs{W^t_{j,m}}}\le \tau,\]
\item for any $n\in A$, any $j\in f^n_t([l])\cap f_s^n([l])$ with $s\ne t$, 
\[
\left|\frac{W^s_{f_s^{-n}(j),n}}{W^t_{f_t^{-n}(j),n}}-\frac{a_{s,f_s^{-n}(j)}}{a_{t,f_t^{-n}(j)}}\right|\le \tau.
\]
\end{enumerate}

Since $\mathcal{A}$ is $(f_s)_{s=1}^N$-weakly partition regular, we can assume that
\[
f_s^m([l])\cap f_t^n([l])=\emptyset  \qquad \forall m, n\in A, m\neq n, \forall s, t\in [N].
\]
%and that for any $j\in [l]$, any $s\in [N]$, $\max_{n\in A}|W^s_{j, n}|\le \tau$.

We then let
\[x=\sum_{n\in A}\sum_{t=1}^{N}\sum_{j\in J_{n,t}}\frac{1}{W^t_{j,n}}a_{t,j}e_{f_t^n(j)}\]
where $J_{n,1}=[l]$ and $J_{n,t+1}=[l]\backslash f_{t+1}^{-n}(\bigcup_{s\le t}f_{s}^n([l]))$ for every $t\in [N-1]$.  
By the definition of sets $J_{n,s}$ and our assumptions on $A$, we get that for every $m\in A$, $n\in A$, if $m\ne n$ or $t\ne s\in [N]$ then 
 \[f^m_t(J_{m,t})\cap f^n_{s}(J_{n,s})=\emptyset.\]
It follows that the vector $x$ is well-defined since
\begin{align*}
\|x\|
&=\max_{n\in A, t\in [N], j\in J_{n,t}}\frac{1}{|W^t_{j,n}|}|a_{t,j}|\\
&\le (\max_{s\in [N]}\|a_s\|)\max_{n\in A, t\in [N], j\in [l]}\frac{1}{|W^t_{j,n}|}<\infty.
\end{align*}
It remains to show that  for every $n\in A$, every $s\in [N]$, $\|T^n_sx-a_s\|<\varepsilon$.
Let $n\in A$ and $s\in [N]$. Note that for every $m\ne n$,
$f_s^{-n}(f_t^m([l]))\cap [l]=\emptyset$ and that if $i\in [l]$ then there are unique $t_i\in [N]$ and $j_i\in J_{n,t_i}$ such that $f_s^{-n}(f^n_{t_i}(j_i))=i$. Therefore, we have 
\begin{align*}
\|T^n_sx-a_s\|&=\|\sum_{m\in A}\sum_{t=1}^{N}\sum_{j\in J_{m,t}}\frac{1}{W^t_{j,m}}a_{t,j}T^n_se_{f_t^m(j)}-a_s\|\\
&\le \left\|\sum_{i=1}^l(\frac{W^s_{i,n}}{W^{t_i}_{j_i,n}}a_{t_i,j_i}e_i-a_{s,i}e_i)\right\|\\
&\quad + \|\sum_{m\in A}\sum_{t=1}^{N}\sum_{j\in J_{m,t}:f_t^m(j)\in f_s^n((l,+\infty))}a_{t, j}\frac{W^s_{f_s^{-n}(f_t^m(j)),n}}{W^t_{j,m}}e_{f_s^{-n}(f_t^m(j))}\|\\
&\le \max_{i\in [l]}|\frac{W^s_{i,n}}{W^{t_i}_{j_i,n}}a_{t_i,j_i}-a_{s,i}|\\
&\quad + (\max_{s\in [N]}\|a_s\|)\max_{t\in [N],j\in [l], m\in A:f_t^m(j)\in f_s^n((l,+\infty))}|\frac{W^s_{f_s^{-n}(f_t^m(j)),n}}{W^t_{j,m}}e_{f_s^{-n}(f_t^m(j))}|\\
&\le \max_{i\in [l]}|\frac{W^s_{i,n}}{W^{t_i}_{j_i,n}}a_{t_i,j_i}-a_{s,i}|+ (\max_{s\in [N]}\|a_s\|)\tau.
\end{align*}
Finally, for any $i\in [l]$, if we let $k_i:=f_{t_i}^{n}(j_i)=f_s^{n}(i)$, we have
\begin{align*}
\Big|\frac{W^s_{i,n}}{W^{t_i}_{j_i,n}}a_{t_i,j_i}-a_{s,i}\Big|&=
\Big|\frac{W^s_{f_{s}^{-n}(k_i),n}}{W^{t_i}_{f_{t_i}^{-n}(k_i),n}}a_{t_i,f_{t_i}^{-n}(k_i)}-a_{s,f_{s}^{-n}(k_i)}\Big|\\
&\le (\max_{s\in [N]}\|a_s\|) \Big|\frac{W^s_{f_{s}^{-n}(k_i),n}}{W^{t_i}_{f_{t_i}^{-n}(k_i),n}}-\frac{a_{s,f_{s}^{-n}(k_i)}}{a_{t_i,f_{t_i}^{-n}(k_i)}}\Big|\\
&\le (\max_{s\in [N]}\|a_s\|)\tau
\end{align*}
and we get the desired result by considering $\tau$ sufficiently small.
\end{proof}

Thanks to Theorem \ref{simpcaraclp} and Theorem \ref{simpcaracc0}, we get a simplified characterization of disjoint reiterative and upper frequent hypercyclicity for any family of weighted shifts. Indeed, for weighted shifts, we have $f^m_s([l])=[m+1, m+l]$ for any $s\in [N]$, and as remarked before $\A_{uBd}$ and $\A_{ud}$ are $(f_s)_{s=1}^N$-weakly partition regular, hence in particular $(f_s)_{s=1}^N$-almost partition regular. We state this characterization in the case of $\ell^p(\mathbb{N})$.
%THEOREM

\begin{theorem}\label{shiftupper}
Let $T_1=B_{w_1},\dots,T_N=B_{w_N}$ be unilateral weighted shifts on $\ell^p(\mathbb{N})$ with $1\le p<\infty$ and $\mathcal{A} = \A_{uBd}$ or $\A_{ud}$. The family $T_1,\dots,T_n$ are disjoint $\mathcal{A}$-hypercyclic if and only if
\begin{enumerate}
\item for any $j\ge 1$ and any $s\in [N]$, we have  $\sum_{n\ge 1}\frac{1}{|W^s_{j, n}|^p}<\infty$,

\item  for any $l\ge 1$, any $(a_{s,j})_{s\in[N],j\in [l]}$ with $a_{s,j}\ne 0$, any $\varepsilon>0$, there exists $A\in \mathcal{A}$ such that for any $s\ne t\in [N]$, any $n\in A$,
\begin{enumerate}
\item 
for any $j\in [l]$,
\[\sum_{m\in A: m>n} \frac{\abs{W^s_{j+m-n,n}}^p}{\abs{W^t_{j,m}}^p}\le \varepsilon,\]
\item for any $j\in [l]$,
\[
\left|\frac{W^s_{j,n}}{W^t_{j,n}}-\frac{a_{s, j}}{a_{t, j}}\right|\le \varepsilon
\]
\end{enumerate}
\end{enumerate}
where $W^s_{l, n}=\prod_{\nu=1}^nw_{s, l+\nu}$, for all $s\in [N]$.
\end{theorem}

% NEW SECTION
\section{Examples}

As shown in \cite{BaRu} and \cite{BMPP1}, the notions of reiterative hypercyclicity, upper frequent hypercyclicity and frequent hypercyclicity coincide for weighted shifts on $\ell^p(\mathbb{N})$. This is in contrast with the disjoint case as the following two results show.

%THEOREM

\begin{theorem}
There exist two upper frequently hypercyclic weighted shifts $B_v$ and $B_w$ on $\ell^p(\mathbb{N})$ with $1\le p<\infty$ such that
$B_v$ and $B_w$ are disjoint reiteratively hypercyclic but not disjoint upper frequently hypercyclic.  
\end{theorem}
\begin{proof}
Let $n_l=10^l$ for any $l\ge 1$ and let $(y_l)_{l\ge 1}$ be a dense sequence in $\ell^p(\mathbb{N})$ such that
\[\text{deg}(y_l)=l\quad\text{and}\quad \frac{1}{l}\le \min_{i\le l} |y_{l,i}|\le \max_{i\le l} |y_{l,i}|\le l.\]

We consider $T_1=B_v$ and $T_2=B_w$ where for every $n\ge 2$, 
\[v_n=\begin{cases}
1 & \textrm{if $n=n_l+1,$}\\
  \frac{y_{l,n-n_l-1}}{y_{l,n-n_l}} & \textrm{if $n\in [n_l+2,n_l+l]$ and $|y_{l,n-n_l}|\ge |y_{l,n-n_l-1}|,$}\\
  1 & \textrm{if $n\in [n_l+2,n_l+l]$ and $|y_{l,n-n_l}|< |y_{l,n-n_l-1}|,$}\\
 2 & \textrm{otherwise,}
\end{cases}\]
and
\[w_n=\begin{cases}
y_{l,1}\frac{W^{1}_{1,n_l-1}}{W^{2}_{1,n_l-1}} & \textrm{if $n=n_l+1,$}\\
   1 & \textrm{if $n\in [n_l+2,n_l+l]$ and $|y_{l,n-n_l}|\ge |y_{l,n-n_l-1}|,$}\\
  \frac{y_{l,n-n_l}}{y_{l,n-n_l-1}} & \textrm{if $n\in [n_l+2,n_l+l]$ and $|y_{l,n-n_l}|< |y_{l,n-n_l-1}|,$}\\
 3 & \textrm{otherwise,}
\end{cases}\]
where $W^1_{l, n}=\prod_{\nu=1}^nv_{l+\nu}$ and $W^2_{l, n}=\prod_{\nu=1}^n w_{l+\nu}$.
It follows by induction that for every $l \ge 1$, every $i\in [l]$,

\[
\frac{W^{2}_{1,n_l+i-1}}{W^{1}_{1,n_l+i-1}}=y_{l,i}.
\]

We also remark that $B_v$ and $B_w$ are continuous since $|v_n|\le 2$ for every $n$ and $\sup_n |w_n|<\infty$. Indeed, it suffices to notice that if $l\ge 2$,
\begin{align*}
|y_{l,1}|\frac{|W^{1}_{1,n_l-1}|}{|W^{2}_{1,n_l-1}|}&\le l \Big(\frac{2}{3}\Big)^{n_l-n_{l-1}-l+1}\frac{|W^{1}_{1,n_{l-1}+l-2}|}{|W^{2}_{1,n_{l-1}+l-2}|}\\
&=l \Big(\frac{2}{3}\Big)^{n_l-n_{l-1}-l+1}\frac{1}{|y_{l-1,l-1}|}\\
&\le l(l-1)\Big(\frac{2}{3}\Big)^{9 \cdot 10^{l-1}-l+1}.
\end{align*}

Moreover, $B_v$ is upper frequently hypercyclic since 
%\begin{itemize}
%\item for every $n<n_1$, $W^{1}_{1,n}=2^n$ and $W^{2}_{1,n}=3^n$
%\item for every $l\ge 1$, every $n\in [n_1-1,n_1]$,
%\[|W^{1}_{1,n}|\ge 2^{n_1-1}\quad \text{and} |W^{2}_{1,n}|\ge 2^{n_1-1}\]
%\item for every $l\ge 1$, every $n\in [n_l+l,n_{l+1}-1)$,
%\[|W^{1}_{1,n}|=|W^{1}_{1,n_{l}+l-1}|2^{n-n_l-l}\quad \text{and}\quad |W^{2}_{1,n}|= |W^{2}_{1,n_{l}+l-1}|3^{n-n_l-l} \]
%\item for every $l\ge 2$, every $n\in [n_l -1,n_l+l)$,
%\[|W^{1}_{1,n}|\le \frac{1}{l^{n-n_l+1}}|W^{1}_{1,n_{l-1}+l-2}|2^{n_l-n_{l-1}-l}\quad \text{and}\quad 
%|W^{2}_{1,n}|\le \frac{1}{l^{n-n_l+2}}|W^{1}_{1,n_{l}-1}|.\]
%\end{itemize}
%It follows that
 for every $l\ge 1$, every $n\in [n_l,n_{l+1})$,
\[|W^{1}_{1,n}|\ge \Big(\frac{1}{l^2}\Big)^{l^2}2^{n-l^2}\]
and $B_w$ is upper frequently hypercyclic
since for every $l\ge 1$, every $n\in [n_l,n_{l+1})$,
\[|W^{2}_{1,n}|\ge \frac{1}{l}|W^{1}_{1,n_l-1}| \Big(\frac{1}{l^2}\Big)^{l}3^{n-n_l-l-1}\ge
\Big(\frac{1}{l^2}\Big)^{l^2+l+1}2^{n-l^2-2l-2}.
 \]
We have thus
\[\sum_{n\ge 1}\frac{1}{|W^{1}_{1,n}|^p}<\infty\quad \text{and}\quad \sum_{n\ge 1}\frac{1}{|W^{2}_{1,n}|^p}<\infty.\]

We now show that the operators $B_v$ and $B_w$ are disjoint reiteratively hypercyclic thanks to Theorem~\ref{shiftupper}. Let $\varepsilon>0$, $l\ge 1$, $(a_{s,i})_{s\in [2], i\in [l]}$ with $a_{s,i}\ne 0$.
We let $\gamma=\min \{\frac{a_{2,i}}{a_{1,i}}:i\in  [l]\}$ and $\Gamma=\max \{\frac{a_{2,i}}{a_{1,i}}:i\in  [l]\}+1$.

There exists $M\ge 1$ such that the following two conditions hold,
\[\frac{\|w\|^{lp}_{\infty}}{(\frac{\gamma}{2})^p(\min_{i\in [l]}\frac{|W^2_{1,i-1}|^p}{|W^1_{1,i-1}|^p})}\sum_{m\ge M}\frac{1}{|W^{1}_{1,m}|^p}<\varepsilon \mbox{ and }\]

\[ 2^{lp}2^p\Gamma^p \big(\max_{i\in [l]}\frac{\abs{W^2_{1,i-1}}^p}{\abs{W^1_{1,i-1}}^p}\big)\sum_{m\ge M}\frac{1}{|W^2_{1,m}|^p}< 
\varepsilon,\]
where $W^s_{i,0}=1$. Furthermore, by density of the sequence $(y_l)$, there exists an increasing sequence $(L_k)_{k\ge 1}$ such that for every $k\ge 1$, $(L_k-l)/M\ge k$ and for every $0\le r\le k$,
\[\max_{i\le l}\Big|\frac{a_{2,i}}{a_{1,i}}-y_{L_k,rM+i}\frac{W^1_{1,i-1}}{W^2_{1,i-1}}\Big|<\min\{\varepsilon/2,\gamma/2, \varepsilon\gamma^2/4, \Gamma\}.\]
We then have for every $k\ge 1$, every $0\le r\le k$, every $i\in [l]$,
\begin{align*}
\Big|\frac{W^2_{i,n_{L_k}+rM}}{W^1_{i,n_{L_k}+rM}}-\frac{a_{2,i}}{a_{1,i}} \Big|&
=\Big|\frac{W^1_{1,i-1}}{W^2_{1,i-1}}\frac{W^2_{1,n_{L_k}+rM+i-1}}{W^1_{1,n_{L_k}+rM+i-1}}-\frac{a_{2,i}}{a_{1,i}} \Big|\\
&=\Big|\frac{W^1_{1,i-1}}{W^2_{1,i-1}}y_{L_k,rM+i}-\frac{a_{2,i}}{a_{1,i}} \Big|<\varepsilon/2,
\end{align*}
and 
\begin{align*}
\Big|\frac{W^1_{i,n_{L_k}+rM}}{W^2_{i,n_{L_k}+rM}}-\frac{a_{1,i}}{a_{2,i}} \Big|
&=\Big|\frac{W^2_{1,i-1}}{W^1_{1,i-1}}\frac{1}{y_{L_k,rM+i}}-\frac{a_{1,i}}{a_{2,i}} \Big|\\
&=\frac{\Big|\frac{W^1_{1,i-1}}{W^2_{1,i-1}}y_{L_k,rM+i}-\frac{a_{2,i}}{a_{1,i}}\Big|}{\Big|\frac{W^1_{1,i-1}}{W^2_{1,i-1}}y_{L_k,rM+i} \frac{a_{2,i}}{a_{1,i}}\Big|}\\
&\le \frac{2\Big|\frac{W^1_{1, i-1}}{W^2_{1, i-1}}y_{L_k,rM+i}-\frac{a_{2,i}}{a_{1,i}} \Big|}{\gamma^2}<\varepsilon/2.
\end{align*}
Moreover, by letting $A=\{n_{L_k}+rM: k\ge 1,\ 0\le r\le k\}$, we have $\overline{\text{Bd}}(A)>0$. We can also remark that since for every $k\ge 1$, every $0\le r\le k$,
\[\max_{i\le l}\left|\frac{a_{2,i}}{a_{1,i}}-y_{L_k,rM+i}\frac{W^1_{1,i-1}}{W^2_{1,i-1}}\right|<\gamma/2,\]
we have 
\[\inf_{k\ge 1,\  0\le r\le k,\ i\le l}|y_{L_k,rM+i}|\ge \frac{\gamma}{2}\min_{i\in [l]}\abs{\frac{W^2_{1,i-1}}{W^1_{1,i-1}}}\]
and thus for any $i\in [l]$, any $n\in A$, 
\begin{align*}
\sum_{m\in A, m>n}\frac{|W^1_{i+m-n,n}|^p}{|W^2_{i,m}|^p}&\le
\|w\|^{lp}_{\infty}\sum_{m\in A, m>n}\frac{|W^1_{i+m-n,n}|^p}{|W^2_{1,m+i-1}|^p}\\
&\le \frac{\|w\|^{lp}_{\infty}}{\inf_{k\ge 1,\ 0\le r\le k,\ i\le l}|y_{L_k,rM+i}|^p}\sum_{m\in A, m>n}\frac{|W^1_{i+m-n,n}|^p}{|W^1_{1,m+i-1}|^p}\\
&\le \frac{\|w\|^{lp}_{\infty}}{(\frac{\gamma}{2})^p\min_{i\in [l]}\abs{\frac{W^2_{1,i-1}}{W^1_{1,i-1}}}^p}\sum_{m\in A, m>n}\frac{1}{|W^1_{1,i+m-n-1}|^p}\\
&\le \frac{\|w\|^{lp}_{\infty}}{(\frac{\gamma}{2})^p\min_{i\in [l]}\abs{\frac{W^2_{1,i-1}}{W^1_{1,i-1}}}^p}\sum_{m\ge M}\frac{1}{|W^1_{1,m}|^p}\le 
\varepsilon,
\end{align*}
and
\begin{align*}
\sum_{m\in A, m>n}\frac{|W^2_{i+m-n,n}|^p}{|W^1_{i,m}|^p}&\le
2^{lp}\sum_{m\in A, m>n}\frac{|W^2_{i+m-n,n}|^p}{|W^1_{1,m+i-1}|^p}\\
&\le 2^{lp} \sup_{k\ge 1,\ 0\le r\le k,\ i\le l}|y_{L_k,rM+i}|^p\sum_{m\in A, m>n}\frac{|W^2_{i+m-n,n}|^p}{|W^2_{1,m+i-1}|^p}\\
&\le 2^{lp}2^p \Gamma^p \max_{i\in [l]}\abs{\frac{W^2_{1,i-1}}{W^1_{1,i-1}}}^p\sum_{m\in A, m>n}\frac{1}{|W^2_{1,i+m-n-1}|^p}\\
&\le 2^{lp}2^p\Gamma^p \max_{i\in [l]}\abs{\frac{W^2_{1,i-1}}{W^1_{1,i-1}}}^p\sum_{m\ge M}\frac{1}{|W^2_{1,m}|^p}\le 
\varepsilon.
\end{align*}
We can therefore deduce from Theorem~\ref{shiftupper} that $B_v$ and $B_w$ are disjoint reiteratively hypercyclic.\\

It remains to show that $B_v$ and $B_w$ are not disjoint upper frequently hypercyclic.
Note that for every $1\le n<n_1$, $\frac{W^{2}_{1,n}}{W^{1}_{1,n}}=\frac{3^n}{2^n}\ge \frac{3}{2}$ and for every $l\ge 1$, every $n\in (n_l+2l,n_{l+1}-1]$,
\[\frac{|W^{2}_{1,n}|}{|W^{1}_{1,n}|}=\frac{|W^{2}_{1,n_l+l-1}|}{|W^{1}_{1,n_l+l-1}|}\Big(\frac{3}{2}\Big)^{n-n_l-l+1}\ge \frac{1}{l}\Big(\frac{3}{2}\Big)^{l+2}\ge \frac{3}{2}.\]
We get
\[\{n\ge 1: \Big|\frac{W^{2}_{1,n}}{W^{1}_{1,n}}-1\Big|< \frac{1}{2}\}\subset \bigcup_{l\ge 1}[n_l,n_l+2l]\]
and since $\overline{\text{dens}}\Big(\bigcup_{l\ge 1}[n_l,n_l+2l]\Big)=0$, it follows from Theorem~\ref{shiftupper} that the shifts $B_v$ and $B_w$ cannot be disjoint upper frequently hypercyclic.
\end{proof}

% THEOREM

\begin{theorem}
There exist two frequently hypercyclic weighted shifts $B_v$ and $B_w$ on $\ell^p(\mathbb{N})$ such that
$B_v$ and $B_w$ are disjoint upper frequently hypercyclic but not disjoint frequently hypercyclic.  
\end{theorem}
\begin{proof}
Let $(n_{l})_{l\ge 1}$ be a rapidly increasing sequence such that
\[\frac{n_{l}}{n_{l+1}}\to 0,\quad \sup_l l^2\Big(\frac{2}{3}\Big)^{n_{l}-2n_{l-1}}<\infty \quad\text{and}\quad \sum_{l\ge 1} n_l\left(\frac{1}{\Big(\frac{3}{4}\Big)^{\sum_{j\le l}n_j}2^{n_l-\sum_{j< l}n_j}}\right)^p<\infty.\]
Let $(y_l)_{l\ge 1}$ be a dense sequence in $\ell^p(\mathbb{N})$ 
such that \[\text{deg}(y_l)=l\quad\text{and}\quad \frac{1}{l}\le \min_{j\le l} |y_{l,j}|\le \max_{j\le l} |y_{l,j}|\le l.\]
Let $\varphi_1:\mathbb{N}\to \mathbb{N}$ and $\varphi_2:\mathbb{N}\to \mathbb{N}$ such that $\varphi_1(l)=J$ implies $l\geq J$ and for every $J,M\ge 1$, $|\{l: \varphi_1(l)=J \text{ and }\varphi_2(l)=M\}|=\infty$.

Let $(\alpha_j)_{j\ge 1}$ such that $0<\alpha_j\le 1$ and $j^{2\alpha_j}\le \frac{4}{3}$. 
Let $l\ge 1$. If $\varphi_1(l)=J \text{ and }\varphi_2(l)=M$, we let
\[z_l=\sum_{r\in [0, (\alpha_J n_l-J)/(J+M)]\cap \N}\sum_{j=1}^{J} y_{J,j}e_{r(J+M)+j}+\sum_{k\in [1,\alpha_J n_l]\backslash\bigcup_{r\in [0, (\alpha_J n_l-J)/(J+M)]\cap \N}[r(J+M)+1,r(J+M)+J]}e_k.\]
We deduce that 
\[\text{deg}(z_l)=\lfloor \alpha_{\varphi_1(l)}n_l\rfloor \quad\text{and}\quad \frac{1}{\varphi_1(l)}\le \min_{j\le \text{deg}(z_l)} |z_{l,j}|\le \max_{j\le \text{deg}(z_l)} |z_{l,j}|\le \varphi_1(l).\]

We consider $T_1=B_v$ and $T_2=B_w$ where for every $n\ge 2$, 
\[v_n=\begin{cases}
1 & \textrm{if $n=n_l+1$}\\
  \frac{z_{l,n-n_l-1}}{z_{l,n-n_l}} & \textrm{if $n\in [n_l+2,n_l+\text{deg}(z_l)]$ and $|z_{l,n-n_l}|\ge |z_{l,n-n_l-1}|$}\\
  1 & \textrm{if $n\in [n_l+2,n_l+\text{deg}(z_l)]$ and $|z_{l,n-n_l}|< |z_{l,n-n_l-1}|$}\\
 2 & \textrm{otherwise}
\end{cases}.\]

%\[v_n=\begin{cases}
%1 & \textrm{if $n=n_l+1$}\\
%  \frac{y_{\varphi_1(l),n-n_l-1}}{y_{\varphi_1(l),n-n_l}} & \textrm{if $n\in [n_l+2,n_l+\varphi_1(l)]$ and $y_{\varphi_1(l),n-n_l}\ge y_{\varphi_1(l),n-n_l-1}$}\\
%  1 & \textrm{if $n\in [n_l+2,n_l+\varphi_1(l)]$ and $y_{\varphi_1(l),n-n_l}< y_{\varphi_1(l),n-n_l-1}$}\\
%  1 & \textrm{if $n=n_l+r\varphi_1(l)+r\varphi_2(l)+1$ with $1\le r\le (2n_l-\varphi_1(l))/(\varphi_1(l)+\varphi_2(l))$ and $y_{\varphi_1(l),1}< y_{\varphi_1(l),l}$}\\
%  \frac{y_{\varphi_1(l),l}}{y_{\varphi_1(l),1}} & \textrm{if $n=n_l+r\varphi_1(l)+r\varphi_2(l)+1$ with $1\le r\le (2n_l-\varphi_1(l))/(\varphi_1(l)+\varphi_2(l))$ and $y_{\varphi_1(l),1}\ge y_{\varphi_1(l),l}$}\\
%  \frac{y_{\varphi_1(l),n-n_l-1}}{y_{\varphi_1(l),n-n_l}} & \textrm{if $n\in [n_l+r\varphi_1(l)+r\varphi_2(l)+2,n_l+r\varphi_1(l)+r\varphi_2(l)+l]$ with $1\le r\le (2n_l-\varphi_1(l))/(\varphi_1(l)+\varphi_2(l))$ and $y_{\varphi_1(l),n-r\varphi_1(l)-r\varphi_2(l)-n_l}\ge y_{\varphi_1(l),n-r\varphi_1(l)-r\varphi_2(l)-n_l-1}$}\\
%  1 & \textrm{if $n\in [n_l+r\varphi_1(l)+r\varphi_2(l)+2,n_l+r\varphi_1(l)+r\varphi_2(l)+l]$ with $1\le r\le (2n_l-\varphi_1(l))/(l+\varphi_2(l))$ and $y_{\varphi_1(l),n-n_l}< y_{\varphi_1(l),n-n_l-1}$}\\  
%  1 & \textrm{if $n\in [n_l,2n_l]\backslash \bigcup_{0\le r\le (2n_l-\varphi_1(l))/(\varphi_1(l)+\varphi_2(l))}[n_l+r\varphi_1(l)+r\varphi_2(l)+1,n_l+r\varphi_1(l)+r\varphi_2(l)+l]$}\\ 
% 2 & \textrm{otherwise}
%\end{cases}.\]
and
\[w_n=\begin{cases}
z_{l,1}\frac{W^{1}_{1,n_l-1}}{W^{2}_{1,n_l-1}} & \textrm{if $n=n_l+1$}\\
   1 & \textrm{if $n\in [n_l+2,n_l+\text{deg}(z_l)]$ and $|z_{l,n-n_l}|\ge |z_{l,n-n_l-1}|$}\\
  \frac{z_{l,n-n_l}}{z_{l,n-n_l-1}} & \textrm{if $n\in [n_l+2,n_l+\text{deg}(z_l)]$ and $|z_{l,n-n_l}|< |z_{l,n-n_l-1}|$}\\
 3 & \textrm{otherwise}
\end{cases}.\]
where $W^1_{l, n}=\prod_{\nu=1}^nv_{l+\nu}$ and $W^2_{l, n}=\prod_{\nu=1}^n w_{l+\nu}$.
It follows by induction that for every $l \ge 1$, every $i\in [\text{deg}(z_l)]$,

\[
\frac{W^{2}_{1,n_l+i-1}}{W^{1}_{1,n_l+i-1}}=z_{l,i}.
\]

We also remark that $B_v$ and $B_w$ are continuous since $|v_n|\le 2$ for every $n$ and $\sup_n |w_n|<\infty$. Indeed, it suffices to notice that if $l\ge 2$,
\begin{align*}
|z_{l,1}|\frac{|W^{1}_{1,n_l-1}|}{|W^{2}_{1,n_l-1}|}&\le l\Big(\frac{2}{3}\Big)^{n_l-n_{l-1}-\text{deg}(z_{l-1})}\frac{|W^{1}_{1,n_{l-1}+\text{deg}(z_{l-1})-1}|}{|W^{2}_{1,n_{l-1}+\text{deg}(z_{l-1})-1}|}\\
&= l \Big(\frac{2}{3}\Big)^{n_l-n_{l-1}-\text{deg}(z_{l-1})}\frac{1}{|z_{l-1,\text{deg}(z_{l-1})}|}\\
&\le l(l-1)\Big(\frac{2}{3}\Big)^{n_l-2n_{l-1}}.
\end{align*}

Moreover, $B_v$ is frequently hypercyclic. Indeed, 
%\begin{itemize}
%\item for every $n<n_1$, $W^{1}_{1,n}=2^n$ and $W^{2}_{1,n}=3^n$
%\item for every $l\ge 1$, every $n\in [n_1-1,n_1]$,
%\[|W^{1}_{1,n}|\ge 2^{n_1-1}\quad \text{and} |W^{2}_{1,n}|\ge 2^{n_1-1}\]
%\item for every $l\ge 1$, every $n\in [n_l+l,n_{l+1}-1)$,
%\[|W^{1}_{1,n}|=|W^{1}_{1,n_{l}+l-1}|2^{n-n_l-l}\quad \text{and}\quad |W^{2}_{1,n}|= |W^{2}_{1,n_{l}+l-1}|3^{n-n_l-l} \]
%\item for every $l\ge 2$, every $n\in [n_l -1,n_l+l)$,
%\[|W^{1}_{1,n}|\le \frac{1}{l^{n-n_l+1}}|W^{1}_{1,n_{l-1}+l-2}|2^{n_l-n_{l-1}-l}\quad \text{and}\quad 
%|W^{2}_{1,n}|\le \frac{1}{l^{n-n_l+2}}|W^{1}_{1,n_{l}-1}|.\]
%\end{itemize}
%It follows that
 for every $l\ge 1$ and every $n\in [n_l,n_l+\text{deg}(z_l)]\subset [n_l,2n_l]$, we have
\begin{align*}
|W^{1}_{1,n}|&\ge \prod_{j\le l}\Big(\frac{1}{(\varphi_1(j))^2}\Big)^{\text{deg}(z_j)}2^{n_l-\sum_{j< l}\text{deg}(z_j)}\\
&\ge \prod_{j\le l}\Big(\frac{1}{(\varphi_1(j))^2}\Big)^{\alpha_{\varphi_1(j)}n_j}2^{n_l-\sum_{j< l}n_j}\\
&\ge \Big(\frac{3}{4}\Big)^{\sum_{j\le l}n_j}2^{n_l-\sum_{j< l}n_j}.
\end{align*}
 Set $\Gamma_l:=\Big(\frac{3}{4}\Big)^{\sum_{j\le l}n_j}2^{n_l-\sum_{j< l}n_j}$. Hence, by our assumption on $(n_l)_{l\ge 1}$, it follows that
\begin{align*}
\sum_{n\geq 1}\frac{1}{\abs{W^1_{1, n}}^p}&\leq \sum_{l=1}^\infty\sum_{n=n_l}^{n_l+\text{deg}(z_l)}\left(\frac{1}{\Gamma_l}\right)^p+\sum_{l=1}^\infty\sum_{n=n_l+\text{deg}(z_l)+1}^{n_{l+1}-1}\left(\frac{1}{\Gamma_l \cdot 2^{n-n_l-\text{deg}(z_l)}}\right)^p\\
&\leq \sum_{l=1}^\infty n_l\left(\frac{1}{\Gamma_l}\right)^p+\sum_{l=1}^\infty\left(\frac{1}{\Gamma_l}\right)^p\sum_{n=1}^{n_{l+1}-n_l-\text{deg}(z_l)-1}\left(\frac{1}{2^n}\right)^p\\
&< \sum_{l=1}^\infty n_l\left(\frac{1}{\Gamma_l}\right)^p + \sum_{l=1}^\infty \left(\frac{1}{\Gamma_l}\right)^p<\infty.
\end{align*} 

 Similarly, the weighted shift $B_w$ is also frequently hypercyclic
since for every $l\ge 1$, every  $n\in [n_l,n_l+\text{deg}(z_l))]\subset [n_l,2n_l]$,
\begin{align*}
|W^{2}_{1,n}|&\ge \frac{1}{\varphi_1(l)}|W^{1}_{1,n_l-1}| \Big(\frac{1}{(\varphi_1(l))^2}\Big)^{\text{deg}(z_l)-1}\\
&\ge \Big(\frac{3}{4}\Big)^{\sum_{j\le l-1}n_j}2^{n_{l-1}-\sum_{j< l-1}n_j} 2^{n_l-2n_{l-1}}\Big(\frac{3}{4}\Big)^{n_l}\\
 &\ge \Big(\frac{3}{4}\Big)^{\sum_{j\le l}n_j}2^{n_{l}-\sum_{j< l}n_j}=\Gamma_l.
 \end{align*}

We now show that the operators $B_v$ and $B_w$ are disjoint upper frequently hypercyclic thanks to Theorem~\ref{shiftupper}. Let $\varepsilon>0$, $l\ge 1$, $(a_{s,i})_{s\in [2], i\in [l]}$ with $a_{s,i}\ne 0$.
We let $\gamma=\min \{\frac{a_{2,i}}{a_{1,i}}:i\in  [l]\}$ and $\Gamma=\max \{\frac{a_{2,i}}{a_{1,i}}:i\in  [l]\}+1$.

There exists $M\ge 1$ such that the following two conditions hold
\[\frac{\|w\|^{lp}_{\infty}}{(\frac{\gamma}{2})^p(\min_{i\in [l]}\frac{|W^2_{1,i-1}|^p}{|W^1_{1,i-1}|^p})}\sum_{m\ge M}\frac{1}{|W^{1}_{1,m}|^p}<\varepsilon, \mbox{ and }\]

\[2^{lp}2^p\Gamma^p(\max_{i\in [l]}\frac{|W^2_{1,i-1}|^p}{|W^1_{1,i-1}|^p}) \sum_{m\ge M}\frac{1}{|W^{2}_{1,m}|^p}<\varepsilon.\]
where $W^s_{i,0}=1$. Let $J\ge l$ such that
\[\max_{i\le l}\Big|\frac{a_{2,i}}{a_{1,i}}-y_{J,i}\frac{W^1_{1,i-1}}{W^2_{1,i-1}}\Big|<\min\{\varepsilon/2,\gamma/2, \varepsilon\gamma^2/4, \Gamma\}.\]
We consider an increasing sequence $(L_k)_{k\ge 1}$ such that for every $k\ge 1$, $\varphi_1(L_k)=J$ and $\varphi_2(L_k)=M$.
We then have for every $k\ge 1$, every $0\le r\le (\alpha_Jn_{L_k}-J)/(J+M)$, every $i\in [l]$,
\begin{align*}
\Big|\frac{W^2_{i,n_{L_k}+r(J+M)}}{W^1_{i,n_{L_k}+r(J+M)}}-\frac{a_{2,i}}{a_{1,i}} \Big|&
=\Big|\frac{W^1_{1,i-1}}{W^2_{1,i-1}}\frac{W^2_{1,n_{L_k}+r(J+M)+i-1}}{W^1_{1,n_{L_k}+r(J+M)+i-1}}-\frac{a_{2,i}}{a_{1,i}} \Big|\\
&=\Big|\frac{W^1_{1,i-1}}{W^2_{1,i-1}}z_{L_k,r(J+M)+i}-\frac{a_{2,i}}{a_{1,i}} \Big|\\
&=\Big|\frac{W^1_{1,i-1}}{W^2_{1,i-1}}y_{J,i}-\frac{a_{2,i}}{a_{1,i}} \Big|<\varepsilon/2
\end{align*}
and 
\begin{align*}
\Big|\frac{W^1_{i,n_{L_k}+r(J+M)}}{W^2_{i,n_{L_k}+r(J+M)}}-\frac{a_{1,i}}{a_{2,i}} \Big|
&=\Big|\frac{W^2_{1,i-1}}{W^1_{1,i-1}}\frac{1}{z_{L_k,r(J+M)+i}}-\frac{a_{1,i}}{a_{2,i}} \Big|\\
&=\frac{\Big|\frac{W^1_{1,i-1}}{W^2_{1,i-1}}z_{L_k,r(J+M)+i}-\frac{a_{2,i}}{a_{1,i}}\Big|}{\Big|\frac{W^1_{1,i-1}}{W^2_{1,i-1}}z_{L_k,r(J+M)+i} \frac{a_{2,i}}{a_{1,i}}\Big|}\\
&\le \frac{2\Big|\frac{W^1_{1, i-1}}{W^2_{1, i-1}}y_{J,i}-\frac{a_{2,i}}{a_{1,i}} \Big|}{\gamma^2}<\varepsilon/2.
\end{align*}
Moreover, by letting $A=\{n_{L_k}+r(J+M): k\ge 1,\ 0\le r\le (\alpha_Jn_{L_k}-J)/(J+M)\}$, we have $\overline{\text{dens}}(A)>0$. Indeed, for every $k\ge 1$, we have
\begin{align*}
|A\cap [0, (1+\alpha_J) n_{L_k}-J]|&\ge |\{n_{L_k}+r(J+M):0\le r\le (\alpha_J n_{L_k}-J)/(J+M)\}|\\
&\ge \frac{\alpha_J n_{L_k}-J}{J+M},
\end{align*}
and thus, since the sequence $(n_{L_k})_k$ tends to infinity, we have
\begin{align*}
\overline{\text{dens}}(A)&\ge \limsup_k \frac{\alpha_J n_{L_k}-J}{J+M}\frac{1}{(1+\alpha_J) n_{L_k}-J+1}\\
&\ge \limsup_k \frac{\alpha_J n_{L_k}-J}{J+M}\frac{1}{(1+\alpha_J) n_{L_k}}= \frac{\alpha_J}{(1+\alpha_J)(J+M)}>0.
\end{align*}

We can also remark that since for every $i\in [l]$,
\[\max_{i\le l}|\frac{a_{2,i}}{a_{1,i}}-y_{J,i}\frac{W^1_{1,i-1}}{W^2_{1,i-1}}|<\gamma/2,\]
we have 
\[\min_{i\le l}|y_{J,i}|\ge \frac{\gamma}{2}\min_{i\in [l]}\frac{W^2_{1,i-1}}{W^1_{1,i-1}}.\]
Therefore, for any $i\in [l]$, any $n\in A$, 
\begin{align*}
\sum_{m\in A, m>n}\frac{|W^1_{i+m-n,n}|^p}{|W^2_{i,m}|^p}&\le
\|w\|^{lp}_{\infty}\sum_{m\in A, m>n}\frac{|W^1_{i+m-n,n}|^p}{|W^2_{1,m+i-1}|^p}\\
&\le \frac{\|w\|^{lp}_{\infty}}{\min_{i\le l}|y_{J,i}|^p}\sum_{m\in A, m>n}\frac{|W^1_{i+m-n,n}|^p}{|W^1_{1,m+i-1}|^p}\\
&\le \frac{\|w\|^{lp}_{\infty}}{(\frac{\gamma}{2})^p\min_{i\in [l]}\abs{\frac{W^2_{1,i-1}}{W^1_{1,i-1}}}^p}\sum_{m\in A, m>n}\frac{1}{|W^1_{1,i+m-n-1}|^p}\\
&\le \frac{\|w\|^{lp}_{\infty}}{(\frac{\gamma}{2})^p\min_{i\in [l]}\abs{\frac{W^2_{1,i-1}}{W^1_{1,i-1}}}^p}\sum_{m\ge M}\frac{1}{|W^1_{1,m}|^p}\le 
\varepsilon
\end{align*}
and
\begin{align*}
\sum_{m\in A, m>n}\frac{|W^2_{i+m-n,n}|^p}{|W^1_{i,m}|^p}&\le
2^{lp}\sum_{m\in A, m>n}\frac{|W^2_{i+m-n,n}|^p}{|W^1_{1,m+i-1}|^p}\\
&\le 2^{lp} \max_{i\le l}|y_{J,i}|^p\sum_{m\in A, m>n}\frac{|W^2_{i+m-n,n}|^p}{|W^2_{1,m+i-1}|^p}\\
&\le 2^{lp} 2^p \Gamma^p \max_{i\in [l]}\abs{\frac{W^2_{1,i-1}}{W^1_{1,i-1}}}^p\sum_{m\in A, m>n}\frac{1}{|W^2_{1,i+m-n-1}|^p}\\
&\le 2^{lp} 2^p \Gamma^p\max_{i\in [l]}\abs{\frac{W^2_{1,i-1}}{W^1_{1,i-1}}}^p\sum_{m\ge M}\frac{1}{|W^2_{1,m}|^p}\le 
\varepsilon.
\end{align*}
We can thus deduce from Theorem~\ref{shiftupper} that $B_v$ and $B_w$ are disjoint upper frequently hypercyclic.\\

On the other hand, $B_v$ and $B_w$ are not disjoint frequently hypercyclic. Indeed, for every $1\le n<n_1$, $\frac{W^{2}_{1,n}}{W^{1}_{1,n}}=\frac{3^n}{2^n}\ge \frac{3}{2}$ and for every $l\ge 1$, every $n\in (2n_l+l,n_{l+1}-1]$,
\[\frac{|W^{2}_{1,n}|}{|W^{1}_{1,n}|}=\frac{|W^{2}_{1,2n_l-1}|}{|W^{1}_{1,2n_l-1}|}\Big(\frac{3}{2}\Big)^{n-2n_l+1}\ge \frac{1}{\varphi_1(l)}\Big(\frac{3}{2}\Big)^{l+2}\ge \frac{1}{l}\Big(\frac{3}{2}\Big)^{l+2}\ge \frac{3}{2}.\]
We deduce thus that
\[\{n\ge 1: \Big|\frac{W^{2}_{1,n}}{W^{1}_{1,n}}-1\Big|< \frac{1}{2}\}\subset \bigcup_{l\ge 1}[n_l,2n_l+l].\]
Therefore, since $\frac{n_{l}}{n_{l+1}}\to 0$, we get $\underline{\text{dens}}\Big(\bigcup_{l\ge 1}[n_l,2n_l+l]\Big)=0$ and it follows from Theorem~\ref{shiftlp} that $B_v$ and $B_w$ cannot be disjoint frequently hypercyclic.
\end{proof}

Conversely it is not possible to find two reiteratively hypercylic unilateral weighted shifts on $\ell^p(\mathbb{N})$ which are disjoint hypercyclic but not disjoint reiteratively hypercyclic. We show this in the next result.

%THEOREM

\begin{theorem}
\label{theor_alhambra}
Let $T_1=T_{f_1,w_1},T_2=T_{f_2,w_2}$ be unilateral pseudo-shifts on $\ell^p(\mathbb{N})$ with $1\le p<\infty$. If $\A_{uBd}$ is $(f_1,f_2)$-almost partition regular and $f_1=f_2$ then $T_1$, $T_2$ are disjoint reiteratively hypercyclic if and only if $T_1$, $T_2$ are reiteratively hypercyclic and disjoint hypercyclic.
\end{theorem}
\begin{proof}
If $T_1$, $T_2$ are disjoint reiteratively hypercyclic, it is obvious that $T_1$, $T_2$ are reiteratively hypercyclic and disjoint hypercyclic.\\

Suppose now that $T_1=T_{f_1,w_1}$, $T_2=T_{f_2,w_2}$ are reiteratively hypercyclic and disjoint hypercyclic with $f_1=f_2$. We show that $T_1$, $T_2$ are disjoint reiteratively hypercyclic by using Theorem~\ref{simpcaraclp}.\\

Since $T_1$ and $T_2$ are reiteratively hypercyclic, we already know that for any $j\ge 1$ and any $s\in [2]$, we have  $\sum_{n\ge 1}\frac{1}{|W^s_{j, n}|^p}<\infty$. Moreover, since $T_1$, $T_2$ are disjoint hypercyclic, we get that for any $l\ge 1$, any $(b_{s,j})_{s\in[2],j\in [l]}$ with $b_{s,j}\ne 0$, any $\varepsilon>0$, any $N\ge 0$, there exists $n\ge N$ such that for any $s\ne t\in [2]$, any $j\in f^n_t([l])\cap f_s^n([l])=f^n_t([l])$,
\[
\left|\frac{W^s_{f_s^{-n}(j),n}}{W^t_{f_t^{-n}(j),n}}-\frac{b_{s,f_s^{-n}(j)}}{b_{t,f_t^{-n}(j)}}\right|\le \varepsilon
\]

Let $l\ge 1$, $(a_{s,j})_{s\in[2],j\in [l]}$ with $a_{s,j}\ne 0$, and $\varepsilon>0$. We consider 
$\gamma=\min\{\frac{|a_{s,j}|}{|a_{t,j}|}: j \in [l], s\ne t\in [2]\}$ and $M> l$ such that for every $j\le l$, every $s\in[2]$,
\[ \frac{2^p}{\gamma^p} \sum_{m\ge M} \frac{1}{\abs{W^s_{j,m}}^p}\le \varepsilon.\]
Let $k\ge 1$, $L=\max\{f_s^{kM}(j):j\in [l]\}$ and $(b_{s,i})_{s\in [2],i\le L}$ with
$b_{s,f_s^{k'M}(j)}=a_{s,j}\Big(W^s_{j,k'M}\Big)^{-1}$ for every $k'\le k$ and $j\in [l]$ and $b_{s,i}=1$ otherwise. Note that $b_{s,i}$ is well-defined because $M>l$. By assumption, for any $N\ge 0$, there exists $n\ge N$ such that for any $s\ne t\in [2]$, any $j\in f^n_t([L])$,
\[
\Big(\max_{k'\le k, i\in [l]}\frac{W^s_{i,k'M}}{W^t_{i,k'M}}\Big)\left|\frac{W^s_{f_s^{-n}(j),n}}{W^t_{f_t^{-n}(j),n}}-\frac{b_{s,f_s^{-n}(j)}}{b_{t,f_t^{-n}(j)}}\right|\le \min\{\varepsilon,\gamma/2\} .
\]
We can thus consider an increasing sequence $(n_k)_{k}$ such that $n_{k+1} \geq n_k + (k+1) M$ for $k \in \N$ and for any $s\ne t\in [2]$, any $j\in f^n_t([L])$,
\[
\Big(\max_{k'\le k, i\in [l]}\frac{W^s_{i,k'M}}{W^t_{i,k'M}}\Big)\left|\frac{W^s_{f_s^{-n_k}(j),n_k}}{W^t_{f_t^{-n_k}(j),n_k}}-\frac{b_{s,f_s^{-n_k}(j)}}{b_{t,f_t^{-n_k}(j)}}\right|\le \min\{\varepsilon,\gamma/2\}.
\]

Let $A=\{n_k+k'M: k \geq 1, k'\le k\}$. We remark that $\overline{Bd}(A)>0$ and that for every $0\le k'\le k$,  any $j\in f^{n_k+k'M}_t([l])$, we have $j\in f^{n_k}_t(f_t^{k'M}([l]))\subset f^{n_k}_t([L])$ and thus
\begin{align*}
&\left|\frac{W^s_{f_s^{-n_k-k'M}(j),n_k+k'M}}{W^t_{f_t^{-n_k-k'M}(j),n_k+k'M}}-\frac{a_{s,f_s^{-n_k-k'M}(j)}}{a_{t,f_t^{-n_k-k'M}(j)}}\right|\\
&\quad= \left|\frac{W^s_{f_s^{-n_k-k'M}(j),k'M}}{W^t_{f_t^{-n_k-k'M}(j),k'M}}\frac{W^s_{f_s^{-n_k}(j),n_k}}{W^t_{f_t^{-n_k}(j),n_k}}-\Big(\frac{W^s_{f_s^{-n_k-k'M}(j),k'M}}{W^t_{f_t^{-n_k-k'M}(j),k'M}}\Big) \frac{b_{s,f_s^{-n_k}(j)}}{b_{t,f_t^{-n_k}(j)}}\right|\\
&\quad=\frac{|W^s_{f_s^{-n_k-k'M}(j),k'M}|}{|W^t_{f_s^{-n_k-k'M}(j),k'M}|}\left|\frac{W^s_{f_s^{-n_k}(j),n_k}}{W^t_{f_t^{-n_k}(j),n_k}}- \frac{b_{s,f_s^{-n_k}(j)}}{b_{t,f_t^{-n_k}(j)}}\right|\\
&\quad\le \min\{\varepsilon,\gamma/2\}.
\end{align*}
In particular, since $f_1=f_2$, we get for every $m\in A$, every $s\ne t\in [2]$, every $i\in [l]$,
\[\frac{|W^s_{i,m}|}{|W^t_{i,m}|}\ge \gamma/2.\]

Therefore, for any $s\ne t\in [2]$, any $n\in A$, any $1\le j\le l$,
\begin{align*}
\sum_{m\in A: f_t^m(j)\in f_s^n((l,+\infty))} \frac{\abs{W^s_{f_s^{-n}(f_t^m(j)),n}}^p}{\abs{W^t_{j,m}}^p}&\le 
\sum_{m\in A: \ m\ge n+M} \frac{\abs{W^s_{f_s^{m-n}(j),n}}^p}{\abs{W^t_{j,m}}^p}\\
&\le 
\sum_{m\in A: \ m\ge n+M} \frac{\abs{W^s_{j,m}}^p}{\abs{W^s_{j,m-n}}^p\abs{W^t_{j,m}}^p}\\
 &\le \frac{2^p}{\gamma^p}\sum_{m\ge M} \frac{1}{\abs{W^s_{j,m}}^p}\le \varepsilon.
\end{align*}
It follows from Theorem~\ref{simpcaraclp} that $T_1$, $T_2$ are disjoint reiteratively hypercyclic.
\end{proof}

% NEW SECTION
\section{Questions}
 Theorem \ref{theor_alhambra} suggests the following natural question.
\begin{question}
 Is it possible to find two reiteratively hypercylic operators $T_1$ and $T_2$ which are disjoint hypercyclic but not disjoint reiteratively hypercyclic? Can $T_1$ and $T_2$ be taken as unilateral weighted shifts on $c_0(\mathbb{N})$?
\end{question}

We say that finitely many operators are densely disjoint $\mathcal{A}$-hypercyclic if the set of their disjoint $\mathcal{A}$-hypercyclic vectors are dense. If $T_1, \dots, T_N$ are disjoint $\A$-hypercyclic unilateral pseudo shifts on $\ell^p(\N)$ or $c_0(\N)$ for $\mathcal{A} = \mathcal{A}_\infty, \mathcal{A}_{uBd}, \mathcal{A}_{ud}$, or $\mathcal{A}_{ld}$, then it is obvious that $T_1, \dots, T_N$ are densely disjoint $\A$-hypercyclic since the perturbation of  
a disjoint $\mathcal{A}$-hypercyclic vector of $(T_1, \dots, T_N)$ by any vector $z\in c_{00}$ is still a disjoint $\mathcal{A}$-hypercyclic vector of $(T_1, \dots, T_N)$.

In \cite{SaSh14}, Sanders and Shkarin showed that every Banach space supports disjoint hypercyclic operators which fail to be densely disjoint hypercyclic. In \cite{MaPu}, the first and the last authors showed the existence of disjoint frequently hypercyclic operators $T_1,T_2$ such that $T_1 \oplus T_1$ and $T_2 \oplus T_2$ are not disjoint hypercyclic on $X \times X$ (that is, $T_1,T_2$ are not disjoint weakly mixing) and therefore, $T_1,T_2$ do not satisfy the Disjoint Hypercyclicity Criterion. We finish the paper with the following open question.

\begin{question}
If $T_1,T_2$ are disjoint frequently hypercyclic operators, must they be densely disjoint hypercyclic?
\end{question}


\begin{thebibliography}{99}

%\bibitem{BaGr1}  F. Bayart and S. Grivaux, {\it Hypercyclicit\'e: le r\^ole du spectre ponctuel unimodulaire}, C. R. Math. Acad. Sci. Paris {\bf 338} (2004), 703–708.

\bibitem{BaGr2} F. Bayart and S. Grivaux, {\it Frequently hypercyclic operators}, Trans. Amer. Math. Soc. {\bf 358} (2006), 5083–5117.

\bibitem{BaMa} F. Bayart and \'{E}. Matheron, \textit{Dynamics of linear operators}, Cambridge Tracts in Mathematics, 179. Cambridge University Press, Cambridge (2009).

 \bibitem{BaRu} F. Bayart and I. Ruzsa, \textit{Difference sets and frequently hypercyclic weighted shifts}, Ergodic Theory Dynam. Systems {\bf 35} (2015), no.3, 691-709. 
 
\bibitem{Ber} L. Bernal-Gonz\'alez, {\it Disjoint hypercyclic operators}, Studia Math., {\bf 182} (2) (2007), 113-130.


\bibitem{BeMaSa} J. B\`es, \"O. Martin, and R. Sanders, {\it Weighted shifts and disjoint hypercyclicity.}  J. Operator Theory {\bf 72} (1) (2014), 15-40.

 \bibitem{BMPP1} J. B\`{e}s, Q. Menet, A. Peris, and Y. Puig, \textit{Recurrence properties of hypercyclic operators}, Math. Ann. {\bf 366} (2016), 545-572.

\bibitem{BePe07} J. B\`es and A. Peris, {\it Disjointness in hypercyclicity}, J. Math. Anal. Appl., {\bf 336} (2007), 297-315. 

\bibitem{BoGr07} A. Bonilla and K.-G. Grosse-Erdmann, {\it Frequently hypercyclic operators and vectors}, Ergodic Theory Dynam. Systems {\bf 27} (2007), 383–404. Erratum: Ergodic Theory Dynam. Systems {\bf 29} (2009), 1993–1994.

      
 \bibitem{BoGr} A. Bonilla and K.-G. Grosse-Erdmann, \textit{Upper frequent hypercyclicity and related notions}, Rev. Mat. Complut., {\bf 31} (3) (2018), 673-711.

\bibitem{CoMaSa} N. \c{C}olako\u{g}lu, \"O. Martin, and R. Sanders, \textit{Disjoint and simultaneous hypercyclic pseudo-shift operators}, Preprint 2020.

\bibitem{Gr}K.-G. Grosse-Erdmann, \textit{Frequently hypercyclic bilateral shifts}, Glasg. Math. J. 61 (2019), no. 2, 271-286.
				  
\bibitem{GrPe} K.-G. Grosse-Erdmann and A. Peris, \textit{Linear chaos}, Springer, London, 2011.

%\bibitem{HiSt} N. Hindman and D. Strauss, \textit{Algebra in the Stone-\v{C}ech compactification. Theory and applications}, Expositions in Mathematics, 27. De Gruyter, Berlin (1998).
           
%\bibitem{KaGu} P. K. Kamthan and M. Gupta, \textit{Sequences spaces and series}, Marcel Dekker, New York 1981.

\bibitem{MaPu} \"O. Martin and Y. Puig, \textit{Existence of disjoint frequently hypercyclic operators which fail to be disjoint weakly mixing}, J. Math. Anal. Appl., {\bf 500} (1) (2021) 125106.

\bibitem{MaSa16} \"O. Martin and R. Sanders, Disjoint supercyclic weighted shifts, {\it Integr. Equ. Oper. Theory}, {\bf 85} (2) (2016), 191-220.

%\bibitem{Or} W. Orlicz, \textit{\"{U}ber unbedingte Konvergenz in Funktionenr\"{a}umen}, Studia Math 4 (1933), 33-37.

%\bibitem{RoFrGrPe12} M. De la Rosa, L. Frerick, S. Grivaux, and A. Peris, {\it Frequent hypercyclicity, chaos, and unconditional Schauder decompositions}, Isr. J. Math., {\bf 190} (2012), 389–399.

%\bibitem{Ru} W. Rudin, \textit{Functional analysis}, McGraw-Hill, New York 1973.  

%\bibitem{Rus} I. Ruzsa, \textit{On difference sets}, Studia Sci. Math. Hungar. 13 (1978), 319-326. 

\bibitem{Sal} H. N. Salas, {\it Hypercyclic weighted shifts}, Trans. Amer. Math. Soc., {\bf 347} (3) (1995), 993–1004.

\bibitem{SaSh14} R. Sanders and S. Shkarin, {\it Existence of disjoint weakly mixing operators that fail to satisfy the Disjoint Hypercyclicity Criterion}, J. Math. Anal. Appl., {\bf 417} (2014) 834–855.




\end{thebibliography}
\end{document}